\documentclass[12pt]{elsarticle}

\usepackage{mathstyle00}
\usepackage{bm}
\usepackage{verbatim} 
\usepackage[margin=0.95in]{geometry}    

\usepackage{lineno}

\begin{document}
\begin{frontmatter}
\title{Convergence analysis of stochastic structure-preserving schemes for computing effective diffusivity 
	in random flows}
		\author[hku]{Junlong Lyu}
		\ead{u3005480@connect.hku.hk}
		\author[hku]{Zhongjian Wang}
		\ead{ariswang@connect.hku.hk}
		\author[uci]{Jack Xin}
		\ead{jxin@math.uci.edu}
		\author[hku]{Zhiwen Zhang\corref{cor1}}
		\ead{zhangzw@hku.hk}
		
		\address[hku]{Department of Mathematics, The University of Hong Kong, Pokfulam Road, Hong Kong SAR, China.}
		\address[uci]{Department of Mathematics, University of California at Irvine, Irvine, CA 92697, USA.}
		\cortext[cor1]{Corresponding author}

\begin{abstract}

In this paper, we develop efficient stochastic structure-preserving schemes to compute the effective
diffusivity for particles moving in random flows. We first introduce the motion of a passive tracer 
particle in random flows using the Lagrangian formulation, which is modeled by stochastic differential equations (SDEs).  Then, we propose stochastic structure-preserving schemes to solve the SDEs and provide rigorous convergence analysis for the numerical schemes in computing effective diffusivity. The convergence analysis follows a probabilistic approach, which interprets the solution process generated by our numerical schemes as a Markov process. By exploring the ergodicity of the solution process, we obtain a convergence analysis of our method in computing long-time solutions of the SDEs. Most importantly, our analysis result reveals the equivalence of the definition of the effective diffusivity by solving discrete-type and continuous-type (i.e. Eulerian) corrector problems, which is fundamental and interesting.  Finally, we present numerical results to demonstrate the accuracy and efficiency of the proposed method and investigate the convection-enhanced diffusion phenomenon in two- and three-dimensional incompressible random flows.  \\

\end{abstract}

\end{frontmatter}

\section{Introduction} \label{sec:introduction}
\noindent
\ {Diffusion enhancement in fluid advection has been studied for nearly a century since the pioneering work of Sir G. Taylor \cite{taylor1922diffusion}. It is a fundamental problem to characterize and quantify the large-scale effective diffusion in fluid flows containing complex and turbulent streamlines, which is of great theoretical and practical importance; see e.g. \cite{Fannjiang:94,Fannjiang:97,	Carmona1997homogenization,Yaulandim:1998,Majda:99,
	PavliotisStuart:07} and references therein. Its applications can be found in many physical and engineering sciences, including atmosphere science, ocean science, chemical engineering, and combustion.}

In this paper, we study the diffusion enhancement phenomenon for particles moving in random flows, which is described by the following passive tracer model, i.e., a stochastic differential equation (SDE) with a random drift,

\begin{align}\label{eqn:generalSDEDefD}
d{\bf X}(t) = {\bf b}(t,{\bf X}(t),\omega)dt + \sigma d{\bf w}(t), \quad  {\bf X}(0) = 0,   
\end{align}
where ${\bf X}(t) \in R^{d}$ is the position of the particle, $\sigma>0$ is the molecular diffusivity, and $\{{\bf w}(t)\}_{t\ge0}$ is the standard $d$-dimensional Brownian motion. Here the velocity field ${\bf b}(t,{\bf x},\omega)$, i.e., the random drift is modeled by a random field in order to mimic the energy spectra of the turbulent flow \cite{kraichnan1970diffusion,Majda:99}. Specifically, we assume ${\bf b}(t,{\bf x},\omega)$ is a zero mean, jointly stationary, ergodic vector random field over a certain probability space, where $\omega$ is an element of the probability space describing all possible environments. The randomness in ${\bf b}(t,{\bf x},\omega)$ is independent of the randomness in the Brownian motion ${\bf w}(t)$. In addition, we assume that the realizations of ${\bf b}(t,{\bf x},\omega)$ are almost surely divergence-free, i.e. $\nabla_{{\bf x}}\cdot {\bf b}(t,{\bf x},\omega) = 0$. To guarantee the existence of the solution to \eqref{eqn:generalSDEDefD}, ${\bf b}(t,{\bf x},\omega)$ should be at least almost surely locally Lipschitz in $\textbf{x}$. To design numerical schemes and carry our convergence analysis, we assume ${\bf b}(t,{\bf x},\omega)$ has certain regularity in the physical space; see Assumption \ref{assump-regularity-b}. We emphasize that since any statement, such as the effective diffusivity, involving statistical properties of the solution ${\bf X}(t)$, requires only convergence in law. Thus, the regularity assumption on the velocity field is natural and will facilitate our algorithm design and convergence analysis in this paper.


We are interested in studying the long-time large-scale behavior of the particles ${\bf X}(t)$ in \eqref{eqn:generalSDEDefD}. Namely, whether the motion of the particles  ${\bf X}(t)$ has a long-time diffusive limit? More specifically, let $X_{\epsilon}(t)\equiv\epsilon X(t/\epsilon^2)$ denote 
the rescaled process of \eqref{eqn:generalSDEDefD}. We want to find conditions under which $X_{\epsilon}(t)$
converges in law, as $\epsilon\rightarrow0$, to a new Brownian motion with a certain covariance matrix $D^{E}\in R^{d\times d}$, where $D^{E}$ is called the effective diffusivity matrix. This problem is referred to as the homogenization of time-dependent flow problem. 
  
\ {Computing the effective diffusivity matrix $D^E$ (i.e., homogenization of time-dependent flows) has been widely studied under various conditions on the flows}. For spatial-temporal periodic velocity fields and random velocity fields with short-range correlations, one can apply the homogenization theory  \cite{BensoussanLionsPapa:2011,Garnier:97,Oleinik:94,Stuart:08} to compute the effective diffusivity matrix $D^E$, where $D^E$ can be expressed in terms of particle ensemble average (Lagrangian framework) or an average of solutions to corrector problems (Eulerian framework). 

\ {The dependence of $D^E$ on the velocity field of the problem is highly nontrivial}. For time-independent Taylor-Green flows, the authors of \cite{StuartZygalakis:09} proposed a stochastic splitting method and calculated the effective diffusivity in the limit of vanishing molecular diffusion. For time-dependent chaotic flows, we proposed a Lagrangian-type numerical integrator to compute the effective diffusivity using structure-preserving schemes \cite{WangXinZhang:18}. In the subsequent work \cite{Zhongjian2018sharp}, we provided a sharp and uniform-in-time error estimate for the numerical integrator in computing the effective diffusivity. However, we point out that the method and the convergence analysis obtained in \cite{WangXinZhang:18,Zhongjian2018sharp} were designated for flows generated from separable and deterministic Hamiltonian only.

\ {For random flows with long-range correlations, the long-time large-scale behavior of the particle motion is complicated and difficult to study in general, since various forms of anomalous diffusion, such as super-diffusion and sub-diffusion may exist}. The interested reader is referred to the review paper \cite{Majda:99}, where  anomalous diffusion was obtained in exactly solvable models. 
See also \cite{eijnden2000generalized} for progress in understanding of the intermittency (i.e., the occurrence of large fluctuations in the velocity field on the small scales) for the passive scalar transport in a turbulent velocity field.

There are several theoretical works on homogenization of time-dependent random flows. Such results include, among others, \cite{Carmona1997homogenization} where the authors proved the existence of the effective diffusivity for a two-dimensional time-dependent incompressible Gaussian velocity field. In \cite{Yaulandim:1998,komorowski2001}, the authors proved the homogenization of convection-diffusion in a time-dependent, ergodic, incompressible random flow. In \cite{fannjiang1996diffusion,fannjiang1999Markovian}, the authors proved some necessary conditions under which the long-time behavior for convection-diffusion in a turbulent flow is diffusive. \ {There are some recent works on studying the effective diffusivity in random flows; see e.g. \cite{boi2015anomalous,chai2016comparative,boi2017eddy,loisy2018effective,roque2018adsorption,renaud2019dispersion}}. Those results show that the dependence of the effective diffusivity upon the molecular diffusion $\sigma$ and the velocity field ${\bf b}$ in the random flow is complicated and how to describe this dependence is very difficult in general. Additionally, it is difficult to study the existence of residual diffusivity for the passive tracer model \eqref{eqn:generalSDEDefD}. The residual diffusivity refers to the non-zero effective diffusivity in the limit of zero molecular diffusion $\sigma$.

\ {This motivates us to develop efficient numerical schemes so that we can compute the effective diffusivity of random flows}. Notice that these random flows are generated from \textit{non-separable Hamiltonian}, which are much more difficult than the problems studied in \cite{WangXinZhang:18}. In this work, we first propose an implicit structure-preserving scheme to solve the SDE \eqref{eqn:generalSDEDefD}, in order to deal with the non-separable Hamiltonian. Second, we provide a \textit{sharp error estimate} for the numerical scheme in computing effective diffusivity. Our analysis is based on a probabilistic approach. We interpret the solution process generated by our numerical scheme as a discrete Markov process, where the transition kernel can be constructed according to the numerical scheme in solving \eqref{eqn:generalSDEDefD}. By exploring the ergodicity of the solution process, we obtain a sharp convergence analysis for our method. Most importantly, our convergence analysis reveals the equivalence of the definition of the effective diffusivity by solving discrete-type and continuous-type (i.e. Eulerian) corrector problems; see Theorem \ref{thm:convergence}, which is fundamental and interesting. Finally, we present numerical results to demonstrate the accuracy of the proposed method in computing effective diffusivity for several incompressible random flows in both two- and three-dimensional space. 

To the best of our knowledge, this paper appears to be the first one in the literature to develop Lagrangian numerical methods to compute effective diffusivity in random flows through the connection with the Eulerian corrector problem. The probabilistic approach in the convergence analysis takes into account the ergodic nature of the solution process and leads to a sharp error estimate. Notice that if one chooses the Gronwall inequality in the error estimate, one cannot get rid of the exponential growth pre-factor in the error term, which makes the estimate not sharp. Moreover, the stochastic structure-preserving Lagrangian scheme enables us to investigate the convection-enhanced diffusion phenomenon in random flows. Especially, we can numerically study the dependence of the effective diffusivity in the regime of small molecular diffusion $\sigma$ and the setting of the velocity field ${\bf b}$ in the random flow.

The rest of the paper is organized as follows. In Section 2, we briefly review some existing results for diffusion in random flows and introduce the definition of effective diffusivity by solving 
a continuous-type corrector problem. In Section 3, we propose our stochastic structure-preserving schemes in computing effective diffusivity for the passive tracer model \eqref{eqn:generalSDEDefD}. In Section 4, we provide the convergence analysis for the proposed method based on a probabilistic approach. In addition, we show the equivalence of the definition of  effective diffusivity through the discrete-type and continuous-type corrector problems. In Section 5, we present numerical results to demonstrate the accuracy and efficiency of our method.  Concluding remarks are made in Section 6.

\section{Preliminaries}\label{sec:preliminaries}
\noindent
To make this paper self-contained, we give a brief review of existing results on convection-enhanced diffusion in random flows and the effective diffusivity. Since these are standard results, we adopt the notations that were used in \cite{fannjiang1996diffusion,fannjiang1999Markovian}. 

\subsection{Some formulations and results for diffusion in random flows} \label{sec:formulations}
\noindent 
\ {We first define a function space that satisfies stationary and ergodic property in $R^d$. Let $(\mathcal{X},\mathcal{H},P_0)$ be a probability space}. Let $\tau_{\bf{x}}$, ${\bf x} \in R^d$ be an almost surely continuous, jointly measurable group of measure preserving transformation on $\mathcal{X}$ with the following properties:
\begin{enumerate}[(T1)]
\item $\tau_{\bf{0}} = Id_{\mathcal{X}}$ and $\tau_{\bf x+y} = \tau_{\bf{x}} \tau_{\bf{y}}$, $\forall~{\bf x},{\bf y} \in R^d$.
\item The mapping $(\chi,\bf{x}) \mapsto \tau_{\bf{x}}\chi$ is jointly measurable.
\item $P_0(\tau_{\bf{x}}(A)) = P_0(A)$, for ${\bf x}\in R^d, A\in \mathcal{H}$.
\item $\lim_{\bf{x}\to0} P_0\big(\chi: |f \circ \tau_{\bf{x}}(\chi) - f(\chi) | \ge \eta\big) = 0,$ $\forall f \in L^2(\mathcal{X})$ and $\forall \eta > 0$.
\item If $P_0\big(A\Delta\tau_{\bf{x}}(A)\big) = 0$,  $\forall~{\bf x} \in R^d$, then $A$ is a trivial event, i.e., $P_0(A)$ is either 0 or 1.
\end{enumerate}
One can verify that $\tau_{\bf x}$ induces a strongly continuous group of unitary mapping $U^{\bf x}$ on $L^2(\mathcal{X})$, which satisfies 

\begin{align}
U^{\bf x}f(\chi) = f(\tau_{\bf x}(\chi)),\ f\in L^2(\mathcal{X}),~{\bf x}\in R^d. 
\end{align}
In addition, it is easily to find that the group $U^{\bf x}$ has $d$ independent, skew-adjoint generators $D_k:\mathcal{D}_k \to L^2(\mathcal{X})$ corresponding to directions $\bf{e}_k$, $k = 1,\cdots,d$.

We introduce some function spaces that are useful in the analysis. Let $C_b^m(\mathcal{X})$ be the space of functions $f$ in the intersection of the domains of $D^\alpha$ with \ {$||D^\alpha f||_{L^\infty(\mathcal{X})}<+\infty$, where $\alpha=(\alpha_1,...,\alpha_d)$ is a multi-index, each component $\alpha_i$ is a nonnegative integer, $\sum_{i=1}^d \alpha_i \le m$, and 
the partial derivative operator $D^\alpha = D_1^{\alpha_1}\circ D_2^{\alpha_2}\circ\cdots\circ D_d^{\alpha_d}$.} It is well known that $C_b^{\infty}(\mathcal{X}) = \cap_{m \ge 1} C_b^m(\mathcal{X})$ is dense in $L^p(\mathcal{X}),1\le p<+\infty$; see \cite{DedikSubin1982}. Let $L^2_0(\mathcal{X}) = \{f\in L^2(\mathcal{X})|{\bf E}_0 f = 0\}$, where ${\bf E}_0$ is the expectation associated with the probability measure $P_0$.

\ {Next, we incorporate the time variable and study the Markov property. Following setting is standard for a general Markov process.}

 Let $\Omega$ be the space of $\mathcal{X}$-valued continuous function $C([0,\infty);\mathcal{X})$ and let $\ell $ be its Borel $\sigma-$algebra. Let $P^t,\ t\ge0$, be a strongly continuous Markov semigroup on $L^2(\mathcal{X})$, which satisfies the following properties.
\begin{enumerate}[(P1)]
\item $P^t \bf{1} = \bf{1}$ and $P^t f \ge 0$, if $f \ge 0$.
\item $\int P^t f dP_0 = \int f dP_0$, for all $f \in L^2(\mathcal{X}),\ t \ge 0$.
\item ${\bf {\bf E}}_{\chi} [f(\theta_{t+h}(\omega))|\ell_{\le t}] = P^h F(\omega(t))$, where $F(\chi) := {\bf E}_\chi f$, for any $f \in L^1(\Omega)$, $t$, $h\ge0$, $\chi \in \mathcal{X}$.
\end{enumerate}
\ {In the property P3, ${\bf E}_{\chi}$ is the expectation associated with the probability measures $P_{\chi}$, which can be considered as the conditional probability for all events in $\ell$ condition on that initial point lies on $\chi$}. $\ell_{\le t}$ are the $\sigma$-algebras generated by events measurable up to time $t$, and $\theta_t(\omega)(\cdot):=\omega(\cdot + t),\  t\ge0$ is the standard shift operator on the path space $(\Omega, \ell)$. 

Moreover, we can define a measure $P$ on the path space $(\Omega, \ell)$ through 

\begin{align}
P(B) = \int P_{\chi} (B) P_0 (d\chi),\ B\in \ell
\end{align}
and define ${\bf E} $ to be the corresponding expectation operator with respect to the measure $P$. As a direct consequence of (T3) and (P2), we know that  $P$ is stationary:
\begin{proposition}\label{P-stationary}
$P$ is invariant under the action of $\theta_t$ and $\tau_{\bf x}$ for any $(t,{\bf x}) \in R^+ \times R^d$.
\end{proposition}
Let $L:\mathcal{D}(L) \to L^2(\mathcal{X})$ be the generator of the semigroup $P^t$. To establish the central limit theorem for the Markov process associated with $P^t$, we assume the generator $L$ satisfies the following time relaxation property, also known as the spectral gap condition,  

\begin{align}
-(Lf,f)_{L^2(\mathcal{X})}\ge c_1||f||^2_{L^2(\mathcal{X})}, ~\text{where}~ c_1>0.
\label{TimeRelaxation}
\end{align} 
The time relaxation property \eqref{TimeRelaxation}	is equivalent to the exponential decay property 

\begin{align}
||P^t f||_{L^2(\mathcal{X})} \le \exp(-2c_1 t) ||f||_{L^2(\mathcal{X})}, ~f\in L^2_0(\mathcal{X}).
\label{exponentialdecayproperty}
\end{align} 
In addition, time relaxation property \eqref{TimeRelaxation} is equivalent to $\rho$-mixing of the process 
$X(t)$, $t\geq 0$. Specifically, let 
$\rho(h) = \sup\{Cor(Y_1,Y_2): Y_1 ~\text{is}~ \ell_{\geq t+h}~\text{measurable}, 
Y_2 ~\text{is}~ \ell_{\leq t}~\text{measurable}\}$, where $Cor(Y_1,Y_2)$ is the correlation function. Then, \eqref{TimeRelaxation} or \eqref{exponentialdecayproperty} implies that $\lim_{h\rightarrow\infty}\rho(h)=0$; see \cite{rosenblatt2012markov,doukhan2012mixing}.
The time relaxation property \eqref{TimeRelaxation} (or the exponential decay property \eqref{exponentialdecayproperty}) plays an important role in proving the existing of the effective diffusivity. We will numerically investigate this property in Section \ref{sec:numericaltest}.

\subsection{The continuous-type corrector problem and effective diffusivity} \label{sec:EffectiveDiffusivity}
\noindent 
Equipped with the necessary properties and notations, we are ready to study the effective diffusivity 
of the random flows associated with the passive tracer model \eqref{eqn:generalSDEDefD}. First we assume that the random flow ${\bf b} = (b_1,...,b_d)\in(L^2(\mathcal{X}))^d$ is jointly continuous in $(t,\bf x)$, locally Lipschitzian in $\bf x$, with finite second moments, and is divergence-free. 


For each fixed realization $\omega$ of the environment, we consider the stochastic process generated by the following SDE,
\begin{equation}\label{omegadiffusion}
\left \{
\begin{aligned}
&d{\bf X}_t^{\omega} = {\bf b}(t,{\bf X}^{\omega}_t,\omega)dt + \sigma d{\bf w}_t,\\
&{\bf X}_0^{\omega} = 0,
\end{aligned}
\right.
\end{equation}
where ${\bf X}^{\omega}_t \in R^{d}$ is the position of the particle, the superscript in ${\bf X}_t^{\omega}$ means that it depends on the realization of the environment $\omega$, \ {${\bf w}_t$ is a standard Brownian motion starting at the origin. Its corresponding probability space is denoted by $(\Sigma, \mathcal{B}, Q)$ and the associated expectation operator is denoted by ${\bf M}$. The SDE \eqref{omegadiffusion} is well-defined \cite{fannjiang1999Markovian}}. Moreover, the random flow in \eqref{omegadiffusion} means ${\bf b}(t,{\bf x},\omega)={\bf b}(\tau_{{\bf x}}\omega(t))$. Viewed from a particle at any instant of time $t$, we can define an 
environment process $\eta:[0,\infty)\times\Omega\to\mathcal{X}$ as  
\begin{equation}\label{environment1}
\left \{
\begin{aligned}
&\eta(t) = \tau_{{\bf X}_t^{\omega}} \omega(t),\\
&\eta(0) = \omega(0).
\end{aligned}
\right.\end{equation}
In addition, environment process generates a semigroup of transformation 
\ {
\begin{equation}
S^t f(\chi) = {\bf M E}_{\chi} f(\eta(t)), ~ t\ge 0, ~  f\in L^{\infty}(\mathcal{X}), 
\label{semigroup-St}
\end{equation}}
where $\eta(t)$ is defined by \eqref{environment1}. And $S^t$ satisfies the following properties,
\begin{proposition}\label{Operator-St}[\cite{fannjiang1999Markovian}, Prop. 3]
\begin{enumerate}[(P1)]
\item $S^t, t\ge 0$ is a strongly continuous, Markov semigroup of contraction on $L^2(\mathcal{X})$.
\item $S^t, t\ge 0$ is measure-preserving, that is,
\begin{equation}
\int S^t f dP_0 = \int f dP_0,~ t\ge 0,~ f \in L^2(\mathcal{X}).
\end{equation}
\end{enumerate}
\end{proposition}
Let $D_1 = \mathcal{D}(L)\cap C_b^2(\mathcal{X})$ and $\mathcal{L}$ denote the generator of the semigroup $S^t$, $t\ge0$, i.e., 
\begin{equation}
  \mathcal{L}f = Lf + \frac{\sigma^2}{2} \Delta f + {\bf b}\cdot \nabla f,
  \label{generateL}
\end{equation} 
where $L$ is the generator of the semigroup $P^t$. One can easily verify the following properties.
\begin{proposition} \label{St-invariant}[\cite{fannjiang1999Markovian},Prop 4]
\begin{enumerate}[(P1)]
 \item $D_1$ is dense in $L^2(\mathcal{X})$ and is invariant under the semigroup $P^t$, $t\ge 0$, 
 i.e., $P^t(D_1) \subseteq D_1$ for all $ t\ge0$.
 \item Assume that the random flow ${\bf b}$ is bounded. Then, $D_1$ is invariant under the semigroup $S^t$, $t \ge 0$, i.e., $S^t(D_1) \subseteq D_1$ for all $t\ge 0$.
\end{enumerate}
\end{proposition} 
\begin{lemma}	 
From the spectral gap condition \eqref{TimeRelaxation}, we obtain that for any $f \in L_0^2 (\mathcal{X})$

\begin{align}
||S^tf||_{L^2(\mathcal{X})} \le \exp(-2c_1 t) ||f||_{L^2(\mathcal{X})},~\text{where}~ c_1>0. 
\label{St-decay}
\end{align}
\end{lemma}
\begin{proof}
We first assume $\bf b$ is bounded and $f \in D_1\subseteq\mathcal{D}(\mathcal{L})$. Using the spectral gap condition and $\bf b$ is divergence-free, we have 
\begin{equation}
	(-\mathcal{L} f, f)_{L^2_0(\mathcal{X})} \ge (-Lf, f)_{L^2_0(\mathcal{X})} \ge c_1||f||^2_{L^2_0(\mathcal{X})} 
\end{equation}
for all $f\in D_1\cap L_0^2(\mathcal{X})$. By Proposition \ref{St-invariant}, $S^t f \in D_1, t\ge 0$ for any $f \in D_1$. Consequently,
\begin{equation}
	\frac{d}{dt} ||S^t f||_{L^2(\mathcal{X})}^2 = 2(\mathcal{L}S^t f, S^t f)_{L^2(\mathcal{X})} \le -2c_1 ||S^t f||_{L^2(\mathcal{X})}^2,
\end{equation}
thus
\ {
\begin{equation}
 ||S^t f||_{L^2(\mathcal{X})}^2\le \exp(-2c_1 t) ||f||_{L^2(\mathcal{X})}^2,\quad  \forall t \ge 0,
\end{equation}}
for all $f \in D_1\cap L_0^2(\mathcal{X})$. Then, the statement in \eqref{St-decay} is extended to $L_0^2(\mathcal{X})$ by using an approximation argument. Finally, the boundedness of the random flow ${\bf b}$ is removed by using another approximation argument.
\end{proof} 
Given the semigroup of transformation $S^t$ in \eqref{semigroup-St} and its associated properties (see Proposition \ref{Operator-St}), we can define
\begin{equation}
{\bm \psi} = \int_0^{\infty} S^t   {\bf b}dt
\label{solution-continue-corrector}
\end{equation}
which satisfies the following continuous-type corrector problem

\begin{align}\label{cell-problem-c}
\mathcal{L} {\bm \psi} = -{\bf b}
\end{align}
where $\mathcal{L}$ is the generator of $S^t$ defined in \eqref{generateL}. By solving the corrector problem
\eqref{cell-problem-c}, we are able to define the effective diffusivity. This can be summarized into the following result.  
\begin{proposition}\label{continus-cellproblem} 
Let ${\bf X}(t)$ be the solution to \eqref{eqn:generalSDEDefD} and ${\bf X}_{\epsilon}(t)\equiv\epsilon {\bf X}(t/\epsilon^2)$. 
For any unit vector ${\bf v} \in R^d$, let $\psi_{\bf v} = \bf \bm\psi\cdot v$ denote 
the projection of the vector solution ${\bm \psi} $ along the direction ${\bf v}$, where ${\bm \psi} $ 
is the solution to corrector problem \eqref{cell-problem-c}. Then, the law of the process ${\bf X}_{\epsilon}(t)\cdot {\bf v}$ converges weakly in $C[0,+\infty)$ to a Brownian motion with diffusion coefficient  given by

\begin{align}
{\bf v}^{T}D^{E}{\bf v} = \frac{\sigma^2}{2} + (-\mathcal{L} \psi_{\bf v}, \psi_{\bf v})_{L^2(\mathcal{X})},
\label{continus-effective-diffusivity}
\end{align}  
where $D^{E}$ is the effective diffusivity associated with the passive tracer model \eqref{eqn:generalSDEDefD}.    
\end{proposition}
The proof of Prop. \ref{continus-cellproblem} relies on an approximation of the additive functional of an ergodic Markov process by a martingale and applying the central limit theorem to continuous-time Markov process, which is very useful in studying the long-time behavior of random dynamics; see  Lemma 1 of \cite{fannjiang1999Markovian} or Theorem of \cite{Carmona1997homogenization}. We shall prove in Theorem \ref{thm:convergence} that the numerical solutions obtained by our Lagrangian numerical scheme recover the definition of the effective diffusivity in \eqref{continus-effective-diffusivity}.

\section{Stochastic structure-preserving schemes and related properties}\label{sec:StructPreserv-schemes}
\subsection{Derivation of numerical schemes}\label{sec:Derivation}
\noindent 
In this section, we construct numerical schemes for the passive tracer model \eqref{omegadiffusion}, 
which is based on an operator splitting method \cite{strang:68}. For each fixed realization $\omega$ of 
the environment, we first split the original problem \eqref{omegadiffusion} into two sub-problems.
\begin{align}
d{\bf X}_t^{\omega} &=  {\bf b}(t,{\bf X}^{\omega}_t,\omega)dt, \label{sub-problem-det}\\
d{\bf X}_t^{\omega} &=   \sigma d{\bf w}_t, \label{sub-problem-rad} 
\end{align}
\ {where we assume ${\bf w}_t$ in \eqref{sub-problem-rad} is the same process as in \eqref{omegadiffusion}.}  Let ${\bf X}^\omega_{n }$ denote the numerical solution of ${\bf X}_t^{\omega}$ at time $t=t_{n}$, $n=0,1,2,...$. From time $t=t_n$ to time $t=t_{n+1}$, where $t_{n+1}=t_{n}+\Delta t$, $t_0=0$, assuming the solution  
${\bf X}^\omega_{n }$ is given, we now discuss how to discretize the above two sub-problems \eqref{sub-problem-det}-\eqref{sub-problem-rad}, separately.

In the sub-problem \eqref{sub-problem-det}, the velocity ${\bf b}(t,{\bf x},\omega)$
is almost surely divergence-free and has certain regularity in the physical space. Thus, we apply a volume-preserving scheme to discretize \eqref{sub-problem-det}. 
Let ${\bf \Phi}_{\Delta t}$ denote the numerical integrator associated with the volume-preserving scheme during $\Delta t$ time and let ${\bf D_{{\bf x}}\Phi}_{\Delta t}$ denote the corresponding Jacobian matrix. The volume-preserving property requests $\det ({\bf D_{{\bf x}}\Phi}_{\Delta t}) = 1$. We obtain
the numerical integrator for the sub-problem \eqref{sub-problem-det} as follows,  

\begin{align}
{\bf X}^\omega_{n+1}={\bf \Phi}^{\omega(n\Delta t)}_{\Delta t}\big({\bf X}^\omega_{n}\big), 
\label{numerical-integrator-sub1}
\end{align} 
where the superscript in ${\bf \Phi}^{\omega(n\Delta t)}_{\Delta t}$ means that the numerical integrator 
implicitly depends on the realization of ${\bf b}$ at different computational times. Suppose $\bf b$ has bounded first derivatives with respect to $\bf x$ for almost all $\omega$, it is easy to verify that the volume-preserving integrator ${\bf\Phi}^{\omega(n\Delta t)}_{\Delta t}$ also has bounded first derivatives for $\Delta t$ small enough. Thus, ${\bf\Phi}^{\omega(n\Delta t)}_{\Delta t}$ is well defined.

In addition, we assume that the numerical scheme only relies on the information of ${\bf X}$ and ${\bf b}$ at the beginning of each computational time,  in order to make sure the solution process generated by our method is a Markov process. For instance, to compute ${\bf X}^\omega_{n+1}$ the numerical scheme only relies on the information of ${\bf X}$ and ${\bf b}$ at $t=t_n$. 

We illustrate this idea by constructing a volume-preserving scheme for a two-dimensional problem. 
Let ${\bf X}^\omega_{n} = (X^\omega_{n,1},X^\omega_{n,2})^{T}$ denote the numerical solution at time $t=t_n$
and the velocity ${\bf b}(t,{\bf x}, \omega) = (b_1(t,{\bf x}, \omega),b_2(t,{\bf x}, \omega))^{T}$.
Then, we use the following numerical scheme to discretize \eqref{sub-problem-det}  
\begin{equation}
 (X^\omega_{{n+1},1},X^\omega_{{n+1},2})^{T} = (X^\omega_{n,1},X^\omega_{n,2})^{T} + \Delta t {\bf b}\big(t_n, (\frac{X^\omega_{n,1}+X^\omega_{{n+1},1}}{2},\frac{X^\omega_{n,2}+X^\omega_{{n+1},2}}{2})^{T},\omega\big),
 \label{twoD-illustation}
\end{equation}
where we evaluate the velocity ${\bf b}(t,{\bf x}, \omega)$ at $t=t_n$ to ensure the Markov property.  By solving Eq.\eqref{twoD-illustation} to get $(X^\omega_{{n+1},1},X^\omega_{{n+1},2})^{T}$, we implicitly define a numerical integrator ${\bf \Phi}^{\omega(n\Delta t)}_{\Delta t}$; see Eq.\eqref{numerical-integrator-sub1}. Since ${\bf b}(t,{\bf x}, \omega)$ is almost surely divergence-free, we can easily verify that the scheme \eqref{twoD-illustation} is volume-preserving, i.e, $\det ({\bf D_{{\bf x}}\Phi}^{\omega(n\Delta t)}_{\Delta t}) = 1$. \ {As we will demonstrate in the proof of Theorem \ref{measure-preserving} that using a volume-preserving numerical scheme  to discretize \eqref{sub-problem-det}  is essential. }


For a $d$-dimensional sub-problem \eqref{sub-problem-det}, we split the velocity field ${\bf b}(t,{\bf x},\omega)$ into a summation of $d-1$ velocity fields, where each of them will generate a two-dimensional problem and thus we can design the volume-preserving scheme accordingly. By applying a splitting method \cite{Quispel:02}, we can construct volume-preserving schemes for the original $d$-dimensional sub-problem \eqref{sub-problem-det}. More details can be found in \cite{KangShang1995volume,ErnstLubich:06}.         

Given the numerical integrator $\bm\Phi^{\omega(n\Delta t)}_{\Delta t}$, we define the mapping

\begin{align}
{\bf B}_{\Delta t}^{\omega(n\Delta t)}({\bf x}) = {\bf \Phi}^{\omega(n\Delta t)}_{\Delta t}({\bf x}) - {\bf x}.\label{Bdefinition}
\end{align} 
One can easily verity that 
${\bf B}_{\Delta t}^{\omega(n\Delta t)}({\bf X}^\omega_{n})$
 is an approximation of the increment for the exact solution of the sub-problem \eqref{sub-problem-det} as follows,

\begin{align}
{\bf X}^\omega_{(n+1)\Delta t}-{\bf X}^\omega_{n\Delta t} = \int_{n\Delta t}^{(n+1)\Delta t} {\bf b}(t,{\bf X}_t^\omega,\omega)dt.
\label{OneStepExactSolution}
\end{align} 

\ {The sub-problem \eqref{sub-problem-rad} can be exactly solved by many numerical schemes for SDEs, including the Euler-Maruyama scheme \cite{Platen:1992}}. 

Finally, we apply the Lie-Trotter splitting method and get the stochastic structure-preserving scheme as follows, 

\begin{align}
{\bf X}^{\omega}_{n+1} = {\bf X}^{\omega}_{n} + {\bf B}_{\Delta t}^{\omega(n\Delta t)}({\bf X}^{\omega}_{n}) + \sigma\bm{\xi}_{n},
\label{one-step}
\end{align}
where $\bm{\xi}_n=(\xi_1,...,\xi_d)^{T}$ is a $d$-dimensional i.i.d. mean-free Gaussian random vector with ${\bf E}{\bm \xi}_n\otimes{\bm \xi}_n  = \Delta t {\bf I}_d$. Here ${\bf I}_d$ is an identity matrix. 
   
The volume-preserving schemes for the sub-problem \eqref{sub-problem-det} are implicit in general.
Compared with explicit schemes, however, they allow us to choose a relatively larger time step to compute. In practice, we find that a few steps of Newton iterations are good enough to maintain accurate results. Therefore, the computational cost is controllable. To design adaptive time-stepping method for the passive tracer model \eqref{omegadiffusion} is an interesting issue, which will be studied in our future work.

In general, the second-order Strang splitting \cite{strang:68} is more frequently used in developing numerical 
schemes. In fact, the only difference between the Strang splitting method and the Lie-Trotter splitting method is that the first and last steps are half of the time step $\Delta t$. For the SDEs, however, the dominant source of error comes from the random subproblem \eqref{sub-problem-rad}. Thus, it is not necessary to implement the Strang splitting scheme here.



\subsection{Some properties of the numerical schemes}\label{sec:propertyNumSche}
\noindent
In this subsection, we shall prove some properties of the proposed stochastic structure-preserving scheme. 
Especially, we shall show that some important properties of the random flows are maintained after numerical discretization. Before proceeding to the analysis, we first introduce some notations and assumptions. 
To emphasize the properties in spatial-domain, for any $f \in L^1(\mathcal{X})$, we use $f^\chi({\bf x})$ to represent $f(\tau_{\bf x} \chi)$. Moreover, we denote ${\bf b}(t,{\bf x},\omega) = {\bf b}(\tau_{\bf x} \omega(t))$, where $\tau_{\bf x} \omega(t) \in \mathcal{X}$. 

\begin{assumption}\label{assump-regularity-b}
Suppose the velocity field has certain regularity with respect to spatial variables, i.e., ${\bf b} \in (C_b^m(\mathcal{X}))^d$ for some $m \ge 1$, and has first-order partial derivative bounded with respect to  temporal variable, i.e., $||{\bf D}_{t}{\bf b}||_{L^\infty(\mathcal{X})}\leq c<\infty$.
\end{assumption}

\begin{assumption}\label{assump-stationary-B}
${\bf B}_{\Delta t}^{\chi}({\bf x}) $ defined in \eqref{Bdefinition} is a stationary process with respect to ${\bf x}$, i.e., we can write ${\bf B}_{\Delta t}^\chi({\bf x}) = {\bf B}_{\Delta t}(\tau_{\bf x} \chi)$.
\end{assumption}

\begin{assumption}\label{assump-regularity-B}
If 	$\Delta t$ is small enough, we have ${\bf B}_{\Delta t} \in (C_b^m(\mathcal{X}))^d$ provided that ${\bf b} \in (C_b^m(\mathcal{X}))^d$. In addition, $||{\bf B}_{\Delta t}||_{C_b^m(\mathcal{X})} = K ||{\bf b}||_{C_b^m(\mathcal{X})} \Delta t$, where $K$ is a constant that does not depend on $\Delta t$.
\end{assumption}
\ {
Under Assumption \ref{assump-regularity-b}, we compute the local truncation error of the numerical scheme \eqref{twoD-illustation}. Recall that the numerical solution ${\bf X}^\omega_{n} = (X^\omega_{n,1},X^\omega_{n,2})^{T}$ at time $t=n\Delta t$. We rewrite \eqref{twoD-illustation} into a compact form as follows,}

\ {
\begin{equation}
{\bf X}^\omega_{n+1}={\bf X}^\omega_{n} + \Delta t{\bf b}\big(t_n,\frac{{\bf X}^\omega_{n}+{\bf X}^\omega_{n+1}}{2},\omega\big).\nonumber
\end{equation}}
\ {
Assume ${\bf X}^\omega_{n}$ is equal to the exact solution ${\bf X}^{\omega}_t$ at time $t=n\Delta t$. Then, we can obtain the exact solution of the sub-problem \eqref{sub-problem-det} at time $t=(n+1)\Delta t$ as
\begin{equation}
{\bf X}^\omega_{(n+1)\Delta t} = {\bf X}^\omega_{n} + \int_{n\Delta t}^{(n+1)\Delta t}{\bf b}(t,{\bf X}_t^\omega,\omega)dt.\nonumber
\end{equation}
Let ${\bf T}^\omega_{n+1}$ denote the local truncation error at time $t=(n+1)\Delta t$. We have}

\ {
\begin{align}
{\bf T}^\omega_{n+1}={\bf X}^\omega_{(n+1)\Delta t}-{\bf X}^\omega_{n+1}=\int_{n\Delta t}^{(n+1)\Delta t}\big( {\bf b}(t,{\bf X}_t^\omega,\omega)-{\bf b}(t_n,\frac{{\bf X}^\omega_{n}+{\bf X}^\omega_{n+1}}{2},\omega)\big)dt. \label{LocalTruncationError} 
\end{align}
We know that ${\bf b}$ satisfies an inequality of the following form }

\ {
\begin{align}
&\big|\big|{\bf b}(t,{\bf X}_t^\omega,\omega)-{\bf b}(t_n,\frac{{\bf X}^\omega_{n}+{\bf X}^\omega_{n+1}}{2},\omega)\big|\big| \nonumber\\
\leq & \big|\big|{\bf b}(t,{\bf X}_t^\omega,\omega)-{\bf b}(t,\frac{{\bf X}^\omega_{n}+{\bf X}^\omega_{n+1}}{2},\omega)\big|\big|+
\big|\big|{\bf b}(t,\frac{{\bf X}^\omega_{n}+{\bf X}^\omega_{n+1}}{2},\omega)-{\bf b}(t_n,\frac{{\bf X}^\omega_{n}+{\bf X}^\omega_{n+1}}{2},\omega)\big|\big|, \nonumber \\
\leq & ||{\bf D_{{\bf x}}}{\bf b}||_{L^\infty(\mathcal{X})}
 \big|\big|{\bf X}_t^\omega-\frac{{\bf X}^\omega_{n}+{\bf X}^\omega_{n+1}}{2}\big|\big|+
||{\bf D}_{t}{\bf b}||_{L^\infty(\mathcal{X})}\big|t-t_n\big|, 
\label{LocalTruncationError2} 
\end{align}
where $t_n\leq t \leq t_{n+1}$ and $||\cdot||$ denotes the Euclidean norm of a vector. If Assumption \ref{assump-regularity-b} holds true,  we can easily obtain that the local truncation error ${\bf T}^\omega_{n+1}=O(\Delta t)^2$, where the constant in $O(\Delta t)^2$ depends on $||{\bf D_{{\bf x}}}{\bf b}||_{L^\infty(\mathcal{X})}$ and $||{\bf D}_{t}{\bf b}||_{L^\infty(\mathcal{X})}$.}

\ {
We restrict ourselves to the convergence analysis based on Assumption \ref{assump-regularity-b} in this paper. In fact, when ${\bf b}$ satisfies a H\"{o}lder-$\gamma$ continuous assumption in time domain with $0<\gamma<1$, the local truncation error of \eqref{twoD-illustation} becomes $O(\Delta t)^{1+\gamma}$.  We can still prove
the convergence analysis of our method for computing effective diffusivity in such kind of flows; see Remark \ref{ConvergenceOrderDependonLemma41}.}

 

As an analogy to the continuous-time case \eqref{environment1}, we define the environment process as viewed from the numerical solution ${\bf X}_n^{\omega}$ at different time steps

\begin{equation}\label{environment2}
\left \{
\begin{aligned}
&\eta_n = \tau_{{\bf X}_n^{\omega}} \omega(n\Delta t),\\
&\eta_0 = \omega(0).
\end{aligned}
\right. \end{equation}
The above environment process is defined on the space of trajectories $(\tilde\Omega, \ell)$, where $\tilde \Omega = C([0,\infty)\cap \Delta t\mathbb{Z};\mathcal{X})$  is a subspace of $\Omega$ with time parameter lies only on $\Delta t \mathbb{Z}$. \ {The corresponding expectation operator is still denoted by ${\bf E}_\chi$, which is the same as the one defined in property P3 in Section \ref{sec:formulations}.}  Under this process, we can write ${\bf B}_{\Delta t}(\eta_n) =  {\bf B}_{\Delta t}^{\omega(n\Delta t)}(X^{\omega}_{n})$. In addition, we define 

\begin{align}
\ {S_nf(\chi) = {\bf M E}_{\chi}f(\eta_n)},
\label{discrete-Sn}
\end{align}
\ {where ${\bf M  }$ denotes the expectation respect to ${\bf w}_t$.} We shall prove that $S_n$ is a discrete-time Markov semi-group of contraction on $L^2(\mathcal{X})$ and is measure-preserving with respect to $P_0$ defined in Section \ref{sec:formulations}. \ {For clarity, we denote $\mathbb{E}$ the total expectation with respect to all randomness, i.e., $\mathbb{E} = {\bf M E}$, in the remaining part of this paper.}
\begin{theorem}\label{measure-preserving}
 $P_0$ is an invariant probability measure of $\eta_n$, i.e., $P_0$ is an invariant measure of the Markov semigroup $\{S_n\}$.
\end{theorem} 
\begin{proof} Let $p^1_{\chi}({\bf x,y})$ denote the transition probability density of the solution process, which is defined by applying the numerical scheme \eqref{one-step} for one time step. 
For simplicity of notation, let ${\bf x}$ be the current solution and ${\bf y}$ be the solution 
obtained by applying the scheme \eqref{one-step} with time step $\Delta t$. 
Notice that $\bm{\xi}_n$ in \eqref{one-step} is a mean-free Gaussian random vector. We have  

\begin{align}
p^1_{\chi}({\bf x,y}) = \frac{1}{(2\pi\sigma^2\Delta t)^{d/2}} \exp\left(-\frac{||{\bf y-x-B}^{\chi}_{\Delta t} ({\bf x})||^2}{2\sigma^2\Delta t}\right)=\frac{1}{(2\pi\sigma^2\Delta t)^{d/2}}\exp\left(-\frac{||{\bf y-\Phi}^{\chi}_{\Delta t} ({\bf x})||^2}{2\sigma^2\Delta t}\right).
\label{TransitionDensity}
\end{align}
Let us define $p_{0}({\bf x,y}) = \frac{1}{(2\pi\sigma^2\Delta t)^{d/2}} \exp\left(-\frac{||{\bf y-x}||^2}{2\sigma^2\Delta t}\right)$. Then, we can verify that 

\begin{align}\label{normalized}
\int p^1_{\chi}({\bf x,y})d{\bf x} &= \int p_0({\bf x+B}^{\chi}_{\Delta t} ({\bf x),y} )d{\bf x}, \nonumber \\
&=\int p_0({\bf z,y})\det({\bf D\Phi}^{\chi}_{\Delta t})^{-1} d{\bf z} =\int p_0({\bf z,y})d{\bf z}= 1 ,~ a.e. ~\chi,
\end{align}
where we have used the fact that the numerical scheme \eqref{numerical-integrator-sub1} for sub-problem \eqref{sub-problem-det} is volume-preserving, i.e., $\det({\bf D\Phi}^{\chi}_{\Delta t})=1$. Thus, for all $f \in L^2(\mathcal{X})$, we have 

\begin{align}\label{ES1fEf}
\int_{\mathcal{X}} S_1 f(\chi) P_0(d\chi) &= \int_{\mathcal{X}} {\bf E}_{\chi} f(\eta_1) P_0(d\chi)  
= \int_{\mathcal{X}} P_0(d\chi) \int_{R^d}   p^1_{\chi}({\bf 0,y}) {\bf E}_{\chi}f(\tau_{\bf y} \omega(\Delta t)) d{\bf y}, \nonumber\\
&= \int_{\mathcal{X}} {\bf E}_{\chi}f(\omega(\Delta t))P_0(d\chi) \int_{R^d}   p^1_{\tau_{\bf -y}\chi}(\bf0,y)dy, \nonumber\\
&= \int_{\mathcal{X}} {\bf E}_{\chi} f(\omega(\Delta t))P_0(d\chi)\int_{R^d}   p^1_{\chi}({\bf -y,0)dy},\nonumber\\
&= \int_{\mathcal{X}} {\bf E}_{\chi} f(\omega(\Delta t))P_0(d\chi),
\end{align}
where we have used the facts that $p^1_{\tau_{\bf x} \chi}({\bf y,z}) = p^1_{\chi}({\bf x+y,x+z})$ and 
$\int_{R^d}   p^1_{\chi}({\bf -y,0)dy}=1$. \ {The first equality is easy to verify, since}

\ {
\begin{align*}
p^1_{\tau_{\bf x} \chi} ({\bf y},{\bf z}) &= p_0 ({\bf y}+{\bf B}_{\Delta t}^{\tau_{\bf x} \chi} ({\bf y}),{\bf z}) = p_0 ({\bf y}+{\bf B}_{\Delta t}^{\chi}({\bf x}+{\bf y}),{\bf z}), \\
&= p_0 ({\bf x}+{\bf y}+{\bf B}_{\Delta t}^{\chi}({\bf x}+{\bf y}),{\bf x}+{\bf z}) = p^1_{\chi}({\bf x}+{\bf y},{\bf x}+{\bf z}).
\end{align*}}
Thus, we obtain from \eqref{ES1fEf} that  
${\bf E } S_1 f = {\bf E } P^{\Delta t}f  = {\bf E} f$, where $P^{\Delta t}$ is measure-preserving by property (P2) in Section \ref{sec:formulations}. Similar argument shows that ${\bf E} S_n f = {\bf E} S_{n-1} f$ for all $n$. We prove that $S_n$ is measure-preserving.  
\end{proof}
\begin{remark}\label{rmk:det1}
Theorem	\ref{measure-preserving} plays an important role in the remaining part of our convergence analysis. 
Throughout the proof, one can see that using a volume-preserving numerical scheme for solving sub-problem \eqref{sub-problem-det} is essential. 
\end{remark}
\begin{remark}\label{rmk:randomness}
In the proof of Theorem \ref{measure-preserving}, the probability measures $p^1_{\chi}({\bf x,y})d{\bf y}$ and 
$p_{0}({\bf x,y})d{\bf y}$ are associated with the Brownian motion in the passive tracer model. While $P_0(d\chi)$ is the probability measure associated with the randomness in the velocity field and initial data. In the remaining part of this paper, we shall keep the same notations. 
\end{remark}
The following lemma will be very useful in our analysis. 
\begin{lemma}\label{lemma} 
For any ${\bf y} \in R^d$ and $f\in L^2(\mathcal{X})$, we have that 

\begin{align}
\mathbb{E} f(\tau_{\bf y}\eta_n) = \mathbb{E} f(\eta_{n-1}) = \mathbb{E} f.
\label{EfEfEf1}
\end{align}
Moreover, 

\begin{align}
 \mathbb{E}f(\eta_{n+1}) = \mathbb{E}f\Big(\tau_{{\bf X}^{\omega}_{n}+{\bf B}_{\Delta t}(\eta_n)}\omega\big((n+1)\Delta t\big)\Big)  =  \mathbb{E}f.
\label{EfEfEf2}
\end{align}
\end{lemma}
\begin{proof} We prove the above equations through direct calculations. For the equation \eqref{EfEfEf1}, we 
have

\begin{align}
&\mathbb{E} f(\tau_{\bf y}\eta_n) = {\bf E M E}_{\eta_{n-1}} f(\tau_{\bf y}\tilde\eta_1)  
= \int_{\mathcal{X}} P_0(d\chi) {\bf M}{\bf E}_\chi\big[\int_{R^d}   p^1_{\eta_{n-1}}({\bf 0,z}) {\bf E}_{\eta_{n-1}}f\big(\tau_{\bf y+z} \omega(\Delta t)\big) d{\bf z}\big] \nonumber\\
=& \int_{\mathcal{X}} {\bf M E}_{\chi}\big[{\bf E}_{\eta_{n-1}}f\big(\omega(\Delta t)\big)P_0(d\chi) \int_{R^d}   p^1_{\tau_{\bf -y-z}\eta_{n-1}}({\bf 0,z})d{\bf z}\big] \nonumber\\
=& \int_{\mathcal{X}} {\bf M E}_{\chi}\big[{\bf E}_{\eta_{n-1}}f\big(\omega(\Delta t)\big)P_0(d\chi) \int_{R^d}   p^1_{\eta_{n-1}}({\bf -y-z,-y})d{\bf z}\big] \nonumber\\
=& \int_{\mathcal{X}} {\bf M E}_{\chi}\big[{\bf E}_{\eta_{n-1}}f\big(\omega(\Delta t)\big)\big]P_0(d\chi)= \int_{\mathcal{X}} {\bf M E}_{\chi}\big[f(\eta_{n-1})\big]P_0(d\chi),
\end{align} 
where $ \tilde\eta_1 $ is defined according to \eqref{environment2} but with initial condition $\tilde\eta_0 = \eta_{n-1}$. Thus, the first equation in \eqref{EfEfEf1} is proved. The second equation in \eqref{EfEfEf1} 
is obvious according to the definition \eqref{discrete-Sn} and $S_n$ is measure-preserving. 

To prove the equation \eqref{EfEfEf2}, let ${\bf Y}_n^{\omega} = {\bf X}^{\omega}_{n}+{\bf B}_{\Delta t}(\eta_n) = {\bf X}^{\omega}_{n+1} - \sigma {\bm \xi}_n$. Then, we have

\begin{align}
\mathbb{E}f(\eta_{n+1}) & = {\bf E M E}_{\eta_{n}} f\big(\tau_{{\bf Y}_n^\omega+\sigma\bm{\xi}_n}\omega(\Delta t)\big) 
=\int_{\mathcal{X}} P_0(d\chi) \int_{R^d}   p_0({\bf 0,z}) {\bf M E}_\chi{\bf E}_{\eta_n}f\big(\tau_{\bf z}\tau_{{\bf Y}_n^\omega} \omega(\Delta t)\big) d{\bf z}, \nonumber\\
&= \int_{\mathcal{X}} {\bf M E}_\chi{\bf E}_{\eta_n}f\big(\tau_{{\bf Y}_n^\omega} \omega(\Delta t)\big)P_0(d\chi) \int_{R^d}   p_0({\bf 0,z})d{\bf z}, \nonumber\\
&=\mathbb{E}f\Big(\tau_{{\bf X}^{\omega}_{n}+{\bf B}_{\Delta t}(\eta_n)}\omega\big((n+1)\Delta t\big)\Big).
\end{align}
Notice that in the proof we use the property that $\tau $ is a measure-preserving transformation. 
\end{proof}
Equipped with these preparations, we can state the main results. The first result is that 
the operator $S_n$ defined in \eqref{discrete-Sn} is a contractive map on $L^2(\mathcal{X})$.
\begin{theorem}\label{thm:exponentialdecay} 
$S_n$ has the property that

\ {
\begin{align}
||S_n f||_{L^2(\mathcal{X})} \le \exp(-2c_1 n\Delta t) ||f||_{L^2(\mathcal{X})},
\label{eq:exponentialdecay}
\end{align}}
for all $f \in L_0^2(\mathcal{X}).$
\end{theorem}
\begin{proof}
We first consider the case when $n=1$. The key observation is that

\begin{align}
&\int_{\mathcal{X}} S_1 f(\chi) \cdot S_1 f(\chi) P_0(d\chi) = \int_{\mathcal{X}} {\bf E}_{\chi} f(\eta_1)\cdot{\bf E}_{\chi} f(\eta_1) P_0(d\chi), \nonumber\\
=& \int_{\mathcal{X}} P_0(d\chi) \int_{R^d}   p^1_{\chi}({\bf 0,y}) {\bf M E}_{\chi}f(\tau_{\bf y} \omega(\Delta t)) d{\bf y} \cdot \int_{R^d}   p^1_{\chi}({\bf 0,y}) {\bf M E}_{\chi}f(\tau_{\bf y} \omega(\Delta t)) d{\bf y}, \nonumber\\
\le&  \int_{\mathcal{X}} P_0(d\chi) \int_{R^d}   p^1_{\chi}({\bf 0,y}) {\bf E}_{\chi}f(\tau_{\bf y} \omega(\Delta t))  \cdot  {\bf E}_{\chi}f(\tau_{\bf y} \omega(\Delta t)) d{\bf y}, \nonumber\\
=&\int_{\mathcal{X}}{\bf E}_{\chi}f(\omega(\Delta t))\cdot {\bf E}_{\chi}f(\omega(\Delta t)) P_0(d\chi) \int_{R^d}   p^1_{\chi}({\bf -y,0})  d{\bf y}, \nonumber\\
=&\int_{\mathcal{X}}P^{\Delta t}f( \chi)\cdot P^{\Delta t}f(\chi) P_0(d\chi), 
\label{Snf_ef}
\end{align}
where $P^{\Delta t}$ is a strongly continuous Markov semigroup on $L^2(\mathcal{X})$. In the third line of \eqref{Snf_ef}, we use the fact that $p^1_{\chi}({\bf 0,y})$ is a probability density function so we can easily get the result by using the Cauchy-Schwarz inequality. 
Therefore, we obtain 

\ {
\begin{align}
||S_1 f||_{L^2(\mathcal{X})} \le ||P^{\Delta t} f||_{L^2(\mathcal{X})} \le \exp(-2c_1 \Delta t) ||f||_{L^2(\mathcal{X})},
\end{align}}
where the exponential decay property \eqref{exponentialdecayproperty} is used. The assertion in \eqref{eq:exponentialdecay} can be obtained if we repeat to use the above property $n$ times.
\end{proof}
Next, we define $\bar {\bf B}_{\Delta t} = \mathbb{E} {\bf B}_{\Delta t}$ and  $\tilde {\bf B}_{\Delta t} = {\bf B}_ {\Delta t}-\bar {\bf B}_{\Delta t}$. We aim to get some estimates for the mean values $\bar {\bf B}_{\Delta t}$ and  $\mathbb{E}{\bf X}^\omega_{n}$, which are important in our convergence analysis
for the effective diffusivity later. 
 
\begin{theorem}\label{EBbound}
\ {Under the Assumptions \ref{assump-regularity-b}, \ref{assump-stationary-B} and \ref{assump-regularity-B}}, if we choose a volume-preserving numerical scheme \eqref{numerical-integrator-sub1} to compute the sub-problem \eqref{sub-problem-det}, where the local truncation error is $O(\Delta t)^2$, then	
$\bar {\bf B}_{\Delta t}$ is of order $O(\Delta t)^2$. In addition, 
$\mathbb{E}{\bf X}^\omega_{n}-n\bar {\bf B}_{\Delta t}$ is bounded.
\end{theorem}
\begin{proof}
By using a volume-preserving numerical scheme (with a local truncation error $O(\Delta t)^2$) to compute \eqref{sub-problem-det}, we have  

\begin{align}
\mathbb{E}{\bf B}_{\Delta t} = \mathbb{E} \int_0^{\Delta t} {\bf b}(t,X_t^\omega,\omega)dt + O(\Delta t)^2=\mathbb{E} \int_0^{\Delta t} {\bf b}(\eta^0_t)dt + O(\Delta t)^2, \label{EBerror}
\end{align} 
where $\eta_t^0$ is the environment process defined in \eqref{environment1} with $\sigma = 0$. \ {Based on the regularity Assumption \ref{assump-regularity-b} for ${\bf b}$, although the constant in the local truncation error $O(\Delta t)^2$ of the numerical scheme \eqref{numerical-integrator-sub1} is random, it has a uniform upper bound. Thus, the error in Eq.\eqref{EBerror} is still of order $O(\Delta t)^2$ after taking the expectation}. Notice that when we define ${\bf B}_{\Delta t}$, we only consider the sub-problem \eqref{sub-problem-det}. Recall the fact that $S^t$ is measure-preserving, so we get

\begin{align}
\mathbb{E} \int_0^{\Delta t} {\bf b}(\eta^0_t)dt = \int_0^{\Delta t} \int_{\mathcal{X}}{\bf E}_{\chi}{\bf b}(\eta^0_t)dP_0(\chi)dt =\int_0^{\Delta t} \mathbb{E}S^t {\bf b}dt =\int_0^{\Delta t} \mathbb{E}{\bf b}dt = 0,
\end{align}
where we have used the definition of $S^t$ in \eqref{semigroup-St} and ${\bf b}$ is mean-zero. 
Therefore, $\mathbb{E} {\bf B}_{\Delta t}$ is of the order $(\Delta t)^2$. Moreover, from the numerical scheme \eqref{one-step} we have 

\begin{align}
\mathbb{E}{\bf X}^\omega_{n} = \mathbb{E}{\bf X}^\omega_{n-1} + \ {\mathbb{E}{\bf B}_{\Delta t}^{\omega((n-1)\Delta t)} ({\bf X}^\omega_{n-1})} = \mathbb{E}{\bf X}^\omega_0 + \sum_{i=0}^{n-1}\mathbb{E}S_i{\bf B}_{\Delta t} = \mathbb{E}{\bf X}^\omega_0 + \sum_{i=0}^{n-1}\mathbb{E}S_i\tilde {\bf B}_{\Delta t} + n\bar {\bf B}_{\Delta t}. 
\end{align}
\ {Under the Assumptions \ref{assump-regularity-b}, \ref{assump-stationary-B} and \ref{assump-regularity-B}, we know that $\bar {\bf B}_{\Delta t}$ and $\tilde {\bf B}_{\Delta t}$ are bounded}. According to \eqref{eq:exponentialdecay} in Theorem \ref{thm:exponentialdecay}, $||S_i\tilde {\bf B}_{\Delta t}||_{L^2(\mathcal{X})}$ decays exponentially with respect to $i$, so we can easily verify that $\sum_{i=0}^{n-1} S_i\tilde {\bf B}_{\Delta t}$ is bounded in $L^2(\mathcal{X})$, which implies $\big|\sum_{i=0}^{n-1}\mathbb{E}S_i\tilde {\bf B}_{\Delta t}\big| < \infty$. Thus, we prove that 
$\mathbb{E}{\bf X}^\omega_{n}-n\bar {\bf B}_{\Delta t}$ is bounded. 
\end{proof}
\subsection{A discrete-type corrector problem}\label{sec:discretetypecorrector}
\noindent 
The corrector problem \eqref{cell-problem-c} plays an important role in defining the effective diffusivity for the random flow. To study the property of the numerical solutions, we will define a discrete-type corrector problem and study the property of its solution.  
\begin{theorem}\label{dcp} 
Let us define ${\bm \psi}_{\Delta t} = \sum_{i=0}^\infty S_i \tilde {\bf B}_{\Delta t}$. Then, 
$\bm\psi_{\Delta t}$ is the unique solution of the discrete-type 
corrector problem in $(L_0^2(\mathcal{X}))^d$ defined as follows

\begin{align}
(S_1 - I) \bm\psi_{\Delta t} =  -\tilde {\bf B}_{\Delta t}.
\label{series}
\end{align}
\end{theorem}
\begin{proof}
The formulation of ${\bm \psi}_{\Delta t}$ solves the discrete-type corrector problem \eqref{series} can be easily verified through simple calculations, i.e.,   

\begin{align}
(S_1 - I) \bm\psi_{\Delta t} = \sum_{i=1}^\infty S_i \tilde {\bf B}_{\Delta t} - \sum_{i=0}^\infty S_i \tilde {\bf B}_{\Delta t} = -\tilde {\bf B}_{\Delta t}.
\end{align}
The property ${\bf E} {\bm \psi}_{\Delta t} = 0$ is a straightforward result from the formulation of ${\bm \psi}_{\Delta t}$. The uniqueness of the solution comes from Theorem \ref{thm:exponentialdecay}. Suppose the equation \eqref{series} has two different solutions $\bm\psi_{1}, \bm\psi_{2}\in L_0^2(\mathcal{X})$, we have that $(S_1 -I)(\bm\psi_{1}-\bm\psi_{2}) = 0$, then 

\begin{align*}
||\bm\psi_{1}-\bm\psi_{2}||_{L^2(\mathcal{X})} = ||S_1 (\bm\psi_{1}-\bm\psi_{2})||_{L^2(\mathcal{X})} \le \exp(-2c_1\Delta t)||\bm\psi_{1}-\bm\psi_{2}||_{L^2(\mathcal{X})},
\end{align*}
which implies that $\bm\psi_{1}-\bm\psi_{2} = 0$. Thus, the uniqueness of solution for Eq.\eqref{series} is proved. 
\end{proof}
\begin{remark}
The formulation of the discrete-type corrector problem \eqref{series} is equivalent to the equation

\begin{align}
\mathbb{E}\big[ \bm\psi^{\omega(i\Delta t)}_{\Delta t}( {\bf X}^\omega_{i})| \eta_{i-1}]- \bm\psi^{\omega((i-1)\Delta t)}_{\Delta t}( {\bf X}^\omega_{i-1})=-\tilde {\bf B}^{\omega((i-1)\Delta t)}_{\Delta t} ( {\bf X}^\omega_{i-1}).
\label{dcp2}
\end{align}
This can be seen by replacing $\chi$ with $\eta_{n-1}$ in the definition of $S_1$; see Eq.\eqref{discrete-Sn}.
\end{remark}
Finally, we study the regularity of the solution of the discrete-type corrector problem \eqref{series}. 
The following result is based on the regularity assumption on the velocity field ${\bf b}$. 
Since we are interested in statistical properties of the solution ${\bf X}(t)$, which only requires convergence in law, we can choose smooth realizations of the velocity field ${\bf b}$. 
\begin{theorem}\label{regularity-cellproblem} 
Suppose ${\bf b} \in (C_b^m(\mathcal{X}))^d$, then $\bm\psi_{\Delta t}$ is in $(H^{m}(\mathcal{X}))^d$.
\end{theorem}
\begin{proof} 
First we prove that, under the assumption ${\bf b} \in (C_b^m(\mathcal{X}))^d$ for $m \ge 1$,
we have that for any $f \in L^2(\mathcal{X})$, $S_1 f\in H^1(\mathcal{X})$. Since

\begin{align}
S_1 f(\tau_{\bf x}\chi) &= \int_{R^d} p_{\tau_{\bf x}\chi}^1({\bf 0,y}) P^{\Delta t}f(\tau_{\bf x+y}\chi)d{\bf y} = \int_{R^d}p_{\chi}^1({\bf x,x+y})P^{\Delta t}f(\tau_{\bf x+y}\chi)d{\bf y}, 
\nonumber  \\
&= \int_{R^d}p_{\chi}^1({\bf x,y})P^{\Delta t}f(\tau_{\bf y}\chi)d\bf y,
\label{S1f}
\end{align}
where $p_{\chi}^1({\bf x,y})$ is the transition probability density defined in \eqref{TransitionDensity}. 
Notice that 

\begin{align}
{\bf D}_{\bf x} p^1_\chi({\bf x,y}) =\big(I + {\bf D}{\bf B}_{\Delta t}^\chi({\bf x})\big)\big({\bf y-x-B}_{\Delta t}^\chi({\bf x})\big)p^1_\chi({\bf x,y})/\ {(\sigma^2 \Delta t)},
\label{Dp1}
\end{align}
and ${\bf B}_{\Delta t} \in (C_b^m(\mathcal{X}))^d$, we can obtain that $\int_{R^d}({\bf y-x-B}^\chi_{\Delta t}({\bf x}))_i^2 p^1_\chi({\bf x,y})d{\bf x}$ is uniformly bounded for almost all $\chi$. Here the indicator $i$ represents the $i-th$ component. This concludes that

\begin{align}
\int_{R^d} {\bf D}_{\bf x}p_\chi^1({\bf x,y}) P^{\Delta t}f(\tau_{\bf y} \chi)d{\bf y} \in (L^2(\mathcal{X}))^d.
\label{Dp-in-L2}
\end{align}
The statement \eqref{Dp-in-L2} implies that ${\bf D} S_1 f \in (L^2(\mathcal{X}))^d$ by the dominant convergence theorem. Thus $S_1 f \in H^1(\mathcal{X})$. According to the definition of the discrete-type corrector problem \eqref{series}, ${\bm\psi}_{\Delta t}$ satisfies 

\begin{align}
{\bm\psi}_{\Delta t} = S_1 {\bm \psi}_{\Delta t} + \tilde {\bf B}_{\Delta t}.
\end{align}
Therefore, we obtain that $\bm\psi_{\Delta t} \in  (H^1(\mathcal{X}))^d$. Moreover, noticing that

\begin{align}
{\bf D}S_1 f(\chi) &= \int_{R^d} {\bf D}_{\bf x}p_\chi^1({\bf 0,y}) P^{\Delta t}f(\tau_{\bf y} \chi)d\bf y, \nonumber\\
&= \int_{R^d}\big(I + {\bf D}{\bf B}_{\Delta t}^\chi(0)\big)\big({\bf y-0-B}_{\Delta t}^\chi(0)\big)p^1_\chi({\bf x,y})P^{\Delta t}f(\tau_y \chi)d{\bf y}/(\sigma^2\Delta t)\nonumber\\
&=\big(I + {\bf D}{\bf B}_{\Delta t}^\chi(0)\big)\int_{R^d}-{\bf D}_{\bf y}p^1_\chi({\bf 0,y})P^{\Delta t}f(\tau_{\bf y} \chi)d{\bf y}/(\sigma^2\Delta t)\nonumber\\
&=\big(I + {\bf D}{\bf B}_{\Delta t}^\chi(0)\big)\int_{R^d}p^1_\chi({\bf 0,y}){\bf D}_{\bf y}P^{\Delta t}f(\tau_{\bf y} \chi)d{\bf y}/(\sigma^2\Delta t)\nonumber\\
&=(\sigma^2\Delta t)^{-1}\big(I + {\bf D}{\bf B}_{\Delta t}^\chi(0)\big)S_1{\bf D}f(\chi).
\end{align}
We arrive that 

\begin{align}
{\bf D}\bm\psi_{\Delta t} = \ {(\sigma^2\Delta t)^{-1}}(I + {\bf D}{\bf B}_{\Delta t})S_1 {\bf D}\bm\psi_{\Delta t} + {\bf D}\tilde {\bf B}_{\Delta t}.
\end{align}
Similar argument shows that ${\bf D}\bm\psi_{\Delta t} \in (H^1(\mathcal{X}))^{d\times d}$. Doing this argument recursively, we prove that $\bm\psi_{\Delta t}$ is in $(H^{m}(\mathcal{X}))^d$.
\end{proof}

\section{Convergence analysis}\label{sec:ConvergenceAnalysis}
\noindent 
In this section, we shall prove the convergence rate of our stochastic structure-preserving scheme in computing effective diffusivity. The convergence analysis is based on a probabilistic approach, which allows us to get rid of the exponential growth factor in the error estimate. 

\subsection{Convergence of the discrete-type corrector problem to the continuous one}\label{sec:Converge-discretocontinuous}
\noindent
We first show that, if $\Delta t$ is small enough, $S^{\Delta t}$ will converge to $S_1$. Moreover, 
the following statement holds. 
\begin{lemma}\label{Sconvergence} 
If $f$ is a globally Lipschitz function with respect to ${\bf x}$, then we have 

\ {
\begin{align}
||S_n f - S^{n\Delta t} f||_{L^2(\mathcal{X})} \le c_2 L \Delta t, 
\label{EulerMethodStrongError0}
\end{align}
}
where $L$ is the Lipschitz constant for $f$ and $c_2$ depends only on the computational time $ n\Delta t$.
\end{lemma}
\begin{proof}
According to the definitions of the semigroups in \eqref{semigroup-St} and \eqref{discrete-Sn}, we have that $(S_n  - S^{n\Delta t}) f(\chi) = {\bf E}_{\chi} \big(f(\eta_n)-f(\eta(n\Delta t))\big)$,
which implies 
 
\begin{align}
(S_n  - S^{n\Delta t}) f(\chi) \le L{\bf E}_{\chi}\big|{\bf X}_n^{\omega}-{\bf X}_{n\Delta t}^{\omega}\big|.
\label{EulerMethodStrongError1}
\end{align}
\ {
The error estimate for the Euler-Maruyama method has been intensively studied in the literature (see e.g. \cite{Platen:1992,milstein1994numerical}). According to Assumption \ref{assump-regularity-b}, the regularity assumption for $\bf b$ is satisfied. Thus, the strong order of accuracy of the Euler-Maruyama scheme for SDE driven by additive noise is 1, i.e., }

\ {
\begin{align}
\sqrt{{\bf E}_{\chi}|{\bf X}_n^{\omega}-{\bf X}_{n\Delta t}^{\omega}|^2} \leq c_2 \Delta t. 
\label{EulerMethodStrongError2}
\end{align} 
The proof is a simple application of Theorem 1.1 in \cite{milstein1994numerical}. 
We apply the Jensen's inequality for expectation and obtain }

\ {
\begin{align}
{\bf E}_{\chi}|{\bf X}_n^{\omega}-{\bf X}_{n\Delta t}^{\omega}|\leq \sqrt{{\bf E}_{\chi}|{\bf X}_n^{\omega}-{\bf X}_{n\Delta t}^{\omega}|^2}\leq c_2 \Delta t. 
\label{EulerMethodStrongError3}
\end{align} 
Combining the estimate results in \eqref{EulerMethodStrongError1} and \eqref{EulerMethodStrongError3}, 
we prove the assertion of Lemma \ref{Sconvergence}. }
	
\end{proof}
Then, we show that under certain conditions the discrete-type corrector problem converges to the continuous one, which facilitates the convergence analysis of our numerical method in computing the effective diffusivity for  random flows.
\begin{theorem}\label{cellconverge}
The solution $\bm\psi_{\Delta t}$ converges to the solution $\bm\psi$ of the continuous-type corrector problem defined in \eqref{solution-continue-corrector} in $L^2(\mathcal{X})$, as $\Delta t\rightarrow 0$.  
\end{theorem}
\begin{proof}
Using the exponential decay properties of $S^t$ and $S_n$, we first choose a truncation time $T_0$ 
and obtain the following two inequalities

\begin{align}
\big |\big |\int_{T_0-\Delta t}^\infty S^t {\bf b} dt\big |\big |_{L^2(\mathcal{X})}\leq \frac{1}{2c_1} \exp(-2c_1 T_0)  , \quad \text{and} \quad 
\big |\big |\sum_{n=[T_0/\Delta t]-1}^\infty S_n\tilde {\bf B}_{\Delta t}\big |\big |_{L^2(\mathcal{X})} \leq \frac{1}{2c_1} \exp(-2c_1 T_0), 
\end{align} 
where $c_1>0$ is defined in \eqref{St-decay}. Then, for any $\epsilon > 0$, we choose $ T_0 $ big enough such that $\frac{1}{c_1} \exp(-c_1 T_0) < \epsilon$. Next, we estimate the error between $\sum_{n=0}^{N-1} S_n\tilde {\bf B}_{\Delta t}$ and $\int_0^{N\Delta t} S^t {\bf b}dt$ for $N \le T_0/\Delta t$. We know that 

\begin{align}
\big |\big |\int_0^{N\Delta t} S^t {\bf b} dt - \sum_{n=0}^{N-1} S^{n\Delta t}{\bf b} \Delta t\big |\big |_{L^2(\mathcal{X})} \le C_1 \Delta t
\end{align}
due to the strongly continuity of $S^t$ (see Prop. \ref{Operator-St}) and

\begin{small}
\begin{align}
\big |\big |\sum_{n=0}^{N-1} S_n\tilde {\bf B}_{\Delta t} - \sum_{n=0}^{N-1} S^{n\Delta t}{\bf b} \Delta t\big |\big |_{L^2(\mathcal{X})} &\le \big |\big |\sum_{n=0}^{N-1} S_n\tilde {\bf B}_{\Delta t} - \sum_{n=0}^{N-1} S_{n}{\bf b} \Delta t\big |\big |_{L^2(\mathcal{X})} \nonumber \\ 
&+ \big |\big |\sum_{n=0}^{N-1} S_n{\bf b}\Delta t - \sum_{n=0}^{N-1} S^{n\Delta t}{\bf b} \Delta t\big |\big |_{L^2(\mathcal{X})}. \label{SBminusSntb}
\end{align}
\end{small}
We can estimate the two terms of the right hand side of the inequality \eqref{SBminusSntb} separately. 
Since the local truncation error of the numerical scheme \eqref{numerical-integrator-sub1} is at least second order, we have $\big |\big |\tilde {\bf B}_{\Delta t} - {\bf b}\Delta t\big |\big |_{L^2(\mathcal{X})}\le O(\Delta t)^2$. From Lemma \ref{Sconvergence}, we know    
$\big |\big |(S_n - S^{n\Delta t}){\bf b}\Delta t\big |\big |_{L^2(\mathcal{X})} \le \ {O(\Delta t)^{2}}$ for all $n \le N$.  This gives the estimate

\ {
\begin{align}
\big |\big |\sum_{n=0}^{N-1} S_n\tilde {\bf B}_{\Delta t} - \sum_{n=0}^{N-1} S^{n\Delta t}{\bf b} \Delta t\big |\big |_{L^2(\mathcal{X})}\le c_2N (\Delta t)^{2} \le c_2T_0 \Delta t.
\end{align}
}
Finally, we take $\Delta t \le \epsilon/(c_2T_0)$ and obtain

\begin{align}
\big |\big |\int_0^{\infty} S^t {\bf b} dt -\sum_{n=0}^{\infty} S_n\tilde {\bf B}_{\Delta t}\big |\big |_{L^2(\mathcal{X})}\le 2\epsilon+O(\epsilon^2).
\end{align}
We prove the assertion of the Theorem. 
\end{proof}
\ {\begin{remark}\label{order}
The constant $c_2$ in Lemma \ref{Sconvergence} actually exponentially depends on $T_0$, i.e., $c_2=\exp(c_3 T_0)$ with $c_3>0$.  To balance each value of $\epsilon$, we have $ \frac{1}{2c_1} \exp(-2c_1 T_0) = \exp(c_3 T_0) T_0 \Delta t$, which requires $T_0 \approx -1/(2c_1+c_3) \log \Delta t$ and $\epsilon \approx \frac{1}{c_1} \Delta t^{\frac{2c_1}{2c_1+c_3}}$.
\end{remark}}

\subsection{Convergence of the numerical method in computing effective diffusivity}\label{sec:converge-DE}
\noindent
Now we are in a position to show the main results of our paper. We prove that the effective diffusivity obtained by our numerical method converges to the exact one defined in \eqref{continus-effective-diffusivity}. 
 
\begin{theorem}\label{thm:convergence}
Let ${\bf X}^{\omega}_n$, $n=0,1,....$ be the numerical solution of  the stochastic structure-preserving scheme \eqref{one-step} and $\Delta t$ be the time step that is fixed. Let $\bar {\bf X}^{\omega}_n = {\bf X}^{\omega}_n - n \bar {\bf B}_{\Delta t} $. We have the convergence estimate of the numerical method in computing effective diffusivity as 

\begin{align}\label{est:order-minusmean}
\frac{\mathbb{E}\bar {\bf X}^{\omega}_n\otimes \bar {\bf X}^\omega_n}{n\Delta t}= \sigma^2 {\bf I}_d +2 {\bf S}\int_{\mathcal{X}} \bm \psi\otimes {\bf b} dP_0 + \rho(\Delta t)+ O(\frac{1}{\sqrt{n}\Delta t}),
\end{align}
where \ {$\rho(\Delta t)=O(\Delta t^{\frac{2c_1}{2c_1+c_3}})$} is a function satisfying $\lim_{\Delta t \to 0}\rho(\Delta t) = 0$ and is independent of the computational time $T$, and $\bf S$ represents the symmetrization operator on a matrix, i.e.,  $\bf SA=\frac{A+A^T}{2}$.
\end{theorem}
\begin{proof}
First of all, from direct computations we obtain that

\begin{small} 
\begin{align}
&\mathbb{E} \bar {\bf X}^\omega_n \otimes \bar {\bf X}^\omega_n = \mathbb{E}\big( \bar {\bf X}^\omega_{n-1} + \tilde {\bf B}^{\omega((n-1)\Delta t)}_{\Delta t} ( {\bf X}^\omega_{n-1}) + \sigma \bm\xi_{n-1}\big) \otimes \big(\bar {\bf X}^\omega_{n-1} + \tilde {\bf B}^{\omega((n-1)\Delta t)}_{\Delta t} ( {\bf X}^\omega_{n-1}) + \sigma \bm\xi_{n-1}\big),\nonumber\\
 =& \mathbb{E} \bar {\bf X}^\omega_{n-1}\otimes \bar {\bf X}^\omega_{n-1} + \sigma^2 {\bf I}_d \Delta t + 2 {\bf S}\mathbb{E} \bar {\bf X}^\omega_{n-1} \otimes \tilde {\bf B}^{\omega((n-1)\Delta t)}_{\Delta t} ( {\bf X}^\omega_{n-1}) + \mathbb{E} \tilde {\bf B}^{\omega((n-1)\Delta t)}_{\Delta t} (\bar {\bf X}^\omega_{n-1}) \otimes \tilde {\bf B}^{\omega((n-1)\Delta t)}_{\Delta t} ( {\bf X}^\omega_{n-1}),\nonumber\\
=&\mathbb{E} \bar {\bf X}^\omega_{0}\otimes \bar {\bf X}^\omega_{0}+ \sigma^2 {\bf I}_d n\Delta t + 2 \sum_{i=1}^{n}{\bf S}\mathbb{E} \bar {\bf X}^\omega_{i-1} \otimes \tilde {\bf B}^{\omega ((i-1)\Delta t)}_{\Delta t} (\bar {\bf X}^\omega_{i-1})+ \sum_{i=1}^{n}\mathbb{E} \tilde {\bf B}^{\omega ((i-1)\Delta t)}_{\Delta t} ( {\bf X}^\omega_{i-1}) \otimes \tilde {\bf B}^{\omega ((i-1)\Delta t)}_{\Delta t} ( {\bf X}^\omega_{i-1}),
\label{EXnXn}
\end{align}
\end{small}
where we have used the facts that $\bm\xi_{n-1}$ is independent with $\bar {\bf X}^\omega_{n-1}$ and $\mathbb{E}{\bm \xi}_{n-1}\otimes{\bm \xi}_{n-1}  = \Delta t {\bf I}_d$. 

The first two terms on the right hand side of Eq.\eqref{EXnXn} are easy to handle since each entry in $\frac{\mathbb{E} \bar {\bf X}^\omega_{0}\otimes \bar {\bf X}^\omega_{0}}{n\Delta t}$ is  $O(\frac{1}{n\Delta t})$ and $\frac{\sigma^2 {\bf I}_d n\Delta t}{n\Delta t}=\sigma^2 {\bf I}_d$. For the forth term on the right hand side of Eq.\eqref{EXnXn}, using the property that $S_i$ is measure-preserving; see Theorem \ref{measure-preserving} and Assumption \ref{assump-regularity-B} , we can get 

\begin{align}
&\frac{1}{n\Delta t}\sum_{i=1}^{n}\mathbb{E}\tilde {\bf B}^{\omega((i-1)\Delta t)}_{\Delta t} ( {\bf X}^\omega_{i-1}) \otimes \tilde {\bf B}^{\omega((i-1)\Delta t)}_{\Delta t} ( {\bf X}^\omega_{i-1}) \nonumber \\
=& \frac{1}{n\Delta t}\sum_{i=1}^{n}\mathbb{E} S_{i-1}(\tilde {\bf B}_{\Delta t}\otimes \tilde {\bf B}_{\Delta t}) = \frac{1}{n\Delta t}n\mathbb{E} \tilde {\bf B}_{\Delta t}\otimes \tilde {\bf B}_{\Delta t}= O(\Delta t).
\label{1overnEBXBX}
\end{align} 
\ {We will focus on the third term on the right hand side of Eq.\eqref{EXnXn}, which corresponds to the strengthen of the convection-enhanced diffusion and is the most difficult term to deal with}. Substituting the formulation of the discrete-type corrector problem \eqref{dcp2} into it, we obtain  

\begin{align}
&\sum_{i = 1}^{n} \mathbb{E} \bar {\bf X}^\omega_{i-1} \otimes \tilde {\bf B}^{\omega((i-1)\Delta t)}_{\Delta t} ( {\bf X}^\omega_{i-1})=-\sum_{i = 1}^{n} \mathbb{E} \bar {\bf X}^\omega_{i-1} \otimes \big(\mathbb{E}[ \bm \psi^{\omega(i\Delta t)}_{\Delta t}( {\bf X}^\omega_{i})| \eta_{i-1}]- \bm \psi^{\omega((i-1)\Delta t)}_{\Delta t}( {\bf X}^\omega_{i-1})\big) \nonumber\\
=&-\sum_{i=1}^{n} \mathbb{E} \mathbb{E}\Big[\bar {\bf X}^\omega_{i-1} \otimes \big( \bm \psi^{\omega(i\Delta t)}_{\Delta t}( {\bf X}^\omega_{i})- \bm \psi^{\omega((i-1)\Delta t)}_{\Delta t}( {\bf X}^\omega_{i-1})\big)\big|\eta_{i-1}\Big]  \nonumber\\
=&-\sum_{i=1}^{n}  \mathbb{E} (\bar {\bf X}^\omega_{i-1} - \bar {\bf X}^\omega_{i})\otimes  \bm \psi^{\omega(i\Delta t)}_{\Delta t}( {\bf X}^\omega_{i}) +\mathbb{E}\bar {\bf X}^\omega_0\otimes  \bm \psi^{\omega(0)}_{\Delta t}( {\bf X}^\omega_{0}) -\mathbb{E} \bar {\bf X}^\omega_n \otimes  \bm \psi^{\omega(n\Delta t)}_{\Delta t}( {\bf X}^\omega_{n})  \nonumber\\
=&\sum_{i=1}^{n}  \mathbb{E} \big(\tilde {\bf B}^{\omega((i-1)\Delta t}_{\Delta t}( {\bf X}^\omega_{i-1}) + \sigma \bm\xi_{i-1}\big)\otimes  \bm \psi^{\omega(i\Delta t)}_{\Delta t}( {\bf X}^\omega_{i}) + \mathbb{E}\bar {\bf X}^\omega_0\otimes  \bm \psi^{\omega(0)}_{\Delta t}( {\bf X}^\omega_{0}) - \mathbb{E} \bar {\bf X}^\omega_n \otimes  \bm \psi^{\omega(n\Delta t)}_{\Delta t}( {\bf X}^\omega_{n}).
\label{sumEXBX}  
\end{align}
\ {Here from the first row to the second row, we use the fact that $\bar{\bf X}^\omega_{i-1}$ and $\bm \psi^{\omega((i-1)\Delta t)}_{\Delta t}( {\bf X}^\omega_{i-1})$ are measurable in the $\sigma$-algebra generated by $\eta_{i-1}$. From the second row to the third row, we use the property of conditional expectation and Abel's summation formula.}

Let us first estimate the summation term on the right hand side of Eq.\eqref{sumEXBX}. For each index $i$, we have

\begin{align}
&\mathbb{E} \big(\tilde {\bf B}^{\omega((i-1)\Delta t)}_{\Delta t}({\bf X}^\omega_{i-1}) + \sigma {\bm\xi}_{i-1}\big)\otimes  \bm \psi^{\omega(i\Delta t)}_{\Delta t}({\bf X}^\omega_{i}),\nonumber\\
=&\mathbb{E} \tilde {\bf B}^{\omega((i-1)\Delta t)}_{\Delta t}({\bf X}^\omega_{i-1})\otimes\bm \psi^{\omega(i\Delta t)}_{\Delta t}({\bf X}^\omega_{i}) + \mathbb{E} \sigma {\bm\xi}_{i-1}\otimes\bm \psi^{\omega(i\Delta t)}_{\Delta t}({\bf X}^\omega_{i}).
\label{EBXplusSigmaPsiX}
\end{align}
Through simple calculations, we can show that the second term of the right hand side of Eq.\eqref{EBXplusSigmaPsiX} is zero. Specifically, we have 

\begin{align}
&\mathbb{E} \sigma {\bm\xi}_{i-1}\otimes\bm \psi^{\omega(i\Delta t)}_{\Delta t}({\bf X}^\omega_{i}) \nonumber\\
=&\mathbb{E} \sigma{\bm\xi}_{i-1}\otimes\bm \psi^{\omega(i\Delta t)}_{\Delta t}\big({\bf X}^\omega_{i-1}+  {\bf B}^{\omega((i-1)\Delta t)}_{\Delta t}({\bf X}^\omega_{i-1}) + \sigma {\bm\xi}_{i-1}\big),\nonumber\\
=&\int_{\mathcal{{X}}}\int_{R^d} p_0({\bf0, y})\sigma{\bf y}\otimes {\bf M E}_\chi\bm \psi_{\Delta t}\big(\tau_{\sigma\bf y} \tau_{{\bf X}^\omega_{i-1}+{\bf B}_{\Delta t}(\eta_{i-1})}\omega(i \Delta t)\big) d{\bf y}P_0(d\chi),\nonumber\\
=&\int_{R^d} p_0({\bf 0, y})\sigma{\bf y}\otimes\int_{\mathcal{{ X}}} {\bf M E}_\chi\bm \psi_{\Delta t}\big(\tau_{\sigma\bf y} \tau_{{\bf X}^\omega_{i-1}+{\bf B}_{\Delta t}(\eta_{i-1})}\omega(i \Delta t)\big) P_0(d\chi)d{\bf y},\nonumber\\
=& \int_{R^d} p_0({\bf 0,y}){\sigma\bf y}\otimes {\bf E}  \bm \psi_{\Delta t} d{\bf y} = 0.
\label{Esigmaxipsi}
\end{align}
Here, the expectation is taken over all the randomness in the system. In the third row of Eq.\eqref{Esigmaxipsi}, ${\bf y}$ is a realization of ${\bm\xi}_{i-1}$ and $p_0({\bf 0,y})d{\bf y}$ is the measure associated with the Brownian motion, while $P_0(d\chi)$ is the measure associated with the randomness in the velocity field and initial data. The Fubini's theorem is used in the fourth row of Eq.\eqref{Esigmaxipsi} to switch the order of integration. The fifth row of \eqref{Esigmaxipsi} is derived from Lemma \ref{lemma}; see Eq.\eqref{EfEfEf2}. And ${\bf E}  \bm \psi_{\Delta t}=0$ since the solution of the discrete-type corrector problem is mean-zero; see Theorem \ref{dcp}.

\ {Then, we compute the first term on the right hand side of Eq.\eqref{EBXplusSigmaPsiX} as follows}, 

\begin{align}
&\mathbb{E} \tilde {\bf B}^{\omega((i-1)\Delta t)}_{\Delta t}({\bf X}^\omega_{i-1})\otimes\bm \psi^{\omega(i\Delta t)}_{\Delta t}({\bf X}^\omega_{i}) = \mathbb{E} \tilde {\bf B}_{\Delta t}(\eta_{i-1})\otimes\bm \psi_{\Delta t}(\eta_i)\nonumber\\
=&\mathbb{E} \mathbb{E}\big(\tilde {\bf B}_{\Delta t}(\eta_{i-1})\otimes \bm \psi_{\Delta t}(\eta_{i})\big|\eta_{i-1}\big)=\mathbb{E} \tilde {\bf B}_{\Delta t}(\eta_{i-1})\otimes \mathbb{E}\big(\bm \psi_{\Delta t}(\eta_{i})\big|\eta_{i-1}\big)\nonumber\\
=&\mathbb{E} \tilde {\bf B}_{\Delta t}(\eta_{i-1})\otimes S_1 \bm \psi_{\Delta t}(\eta_{i-1})=
\mathbb{E} \tilde {\bf B}_{\Delta t}(\eta_{i-1})\otimes  \big(\bm \psi_{\Delta t}(\eta_{i-1}) - \tilde{\bf B}_{\Delta t}(\eta_{i-1})\big)\nonumber\\
=&\mathbb{E} \tilde {\bf B}_{\Delta t}(\eta_{i-1})\otimes \bm \psi_{\Delta t}(\eta_{i-1}) - \mathbb{E} \tilde {\bf B}_{\Delta t}(\eta_{i-1})\otimes\tilde{\bf B}_{\Delta t}(\eta_{i-1})\nonumber\\
=&\mathbb{E} S_{i-1}(\tilde {\bf B}_{\Delta t}\otimes \bm \psi_{\Delta t}) - \mathbb{E} S_{i-1}(\tilde {\bf B}_{\Delta t}\otimes \tilde{\bf B}_{\Delta t}).
\end{align}
Using the property that each $S_{i-1}$ is measure-preserving; see Theorem \ref{measure-preserving}, we have

	\begin{align}\label{1st}
	\frac{1}{n}\sum_{i=1}^{n} \mathbb{E} \tilde {\bf B}^{\omega((i-1)\Delta t)}_{\Delta t}({\bf X}^\omega_{i-1})\otimes\bm \psi^{\omega(i\Delta t)}_{\Delta t}({\bf X}^\omega_{i}) =  \mathbb{E} \tilde {\bf B}_{\Delta t}\otimes \bm \psi_{\Delta t} -  \mathbb{E} \tilde {\bf B}_{\Delta t}\otimes \tilde {\bf B}_{\Delta t}.
	\end{align}
The term $ \mathbb{E} \tilde {\bf B}_{\Delta t}\otimes \bm \psi_{\Delta t}$ in Eq.\eqref{1st} is corresponding to the strengthen of the convection-enhanced diffusion. The term $\mathbb{E} \tilde {\bf B}_{\Delta t}\otimes \tilde {\bf B}_{\Delta t}$ in Eq.\eqref{1st} is of the order $O(\Delta t)^2$ due to Assumption \ref{assump-regularity-B} . \ {This completes the estimate of the first term 
on the right hand side of Eq.\eqref{sumEXBX}}.

\ {Now, we estimate the second term and third term on the right hand side of Eq.\eqref{sumEXBX}}. The term $\mathbb{E}\bar {\bf X}^\omega_0\otimes  \bm \psi^{\omega(0)}_{\Delta t}( {\bf X}^\omega_{0})$ does not depend on $n$ and is bounded. For the third term, we want to prove that 
\begin{equation} \label{3rdbound}
 \frac{1}{n\Delta t}\big|\big|\mathbb{E} \bar {\bf X}^\omega_n \otimes  \bm \psi^{\omega(n\Delta t)}_{\Delta t}( {\bf X}^\omega_{n})\big|\big| \le O(\frac{1}{\sqrt{n}\Delta t}),
\end{equation}
where $||\cdot||$ is a matrix norm. By using the Holder's inequality, we know that each entry of the matrix $\mathbb{E} \bar {\bf X}^\omega_n \otimes  \bm \psi^{\omega(n\Delta t)}_{\Delta t}( {\bf X}^\omega_{n})$ satisfies 
\begin{equation}
\big|\mathbb{E}(\bar {\bf X}^\omega_n)_l  \big(\bm\psi^{\omega(n\Delta t)}_{\Delta t}( {\bf X}^\omega_{n})\big)_j\big| \le \big(\mathbb{E} [(\bar {\bf X}^\omega_n)_l]^2\big)^{1/2} \big(\mathbb{E}[(\bm\psi^{\omega(n\Delta t)}_{\Delta t}( {\bf X}^\omega_{n}))_j]^2\big)^{1/2}, \quad 1\leq l,j \leq d. \label{EXnlPhiXnj}
\end{equation}
Again, using the property that $S_n$ is measure-preserving; see Theorem \ref{measure-preserving}, we have
\begin{equation}
\mathbb{E}\big[(\bm\psi^{\omega(n\Delta t)}_{\Delta t}( {\bf X}^\omega_{n}))_j\big]^2 = 
\mathbb{E} \big(\bm\psi_{\Delta t,j}(\eta_{n})\big)^2 = \mathbb{E} S_n (\bm\psi_{\Delta t,j})^2 = \mathbb{E}(\bm\psi_{\Delta t,j})^2,
\label{EPhiXnEPhi}
\end{equation}
which is bounded since $\bm\psi_{\Delta t}\in (L_0^2(\mathcal{X}))^d$ according to Theorem \ref{dcp}. Thus, if we can prove $\frac{1}{n} \mathbb{E} [(\bar {\bf X}^\omega_n)_l]^2$ is bounded, then  
\begin{equation} 
\frac{1}{n\Delta t}\big|\mathbb{E}(\bar {\bf X}^\omega_n)_l  (\bm\psi^{\omega(n\Delta t)}_{\Delta t}( {\bf X}^\omega_{n}))_j\big| \le \frac{1}{\sqrt{n}\Delta t} (\frac{1}{n}\mathbb{E} [(\bar {\bf X}^\omega_n)_l]^2)^{1/2} (\mathbb{E}[(\bm\psi^{\omega(n\Delta t)}_{\Delta t}( {\bf X}^\omega_{n}))_j]^2)^{1/2} =  O(\frac{1}{\sqrt{n}\Delta t}).
\label{EXnPhiXn}
\end{equation} 

\ {In order to prove that  $\frac{1}{n} \mathbb{E} [(\bar {\bf X}^\omega_n)_l]^2$ is bounded, we apply the AM-GM inequality on the diagonal entries of the matrix $\mathbb{E} \bar {\bf X}^\omega_n \otimes  \bm \psi^{\omega(n\Delta t)}_{\Delta t}( {\bf X}^\omega_{n})$ and obtain},

\begin{align}\label{amgmieq}
\big|\mathbb{E} (\bar {\bf X}^\omega_n)_l  (\bm\psi^{\omega(n\Delta t)}_{\Delta t}( {\bf X}^\omega_{n}))_l\big| \le \epsilon \mathbb{E} [(\bar {\bf X}^\omega_n)_l]^2 + (4\epsilon)^{-1}\mathbb{E}[(\bm\psi^{\omega(n\Delta t)}_{\Delta t}( {\bf X}^\omega_{n}))_l]^2, \quad 1\leq l \leq d,
\end{align}
where $0< \epsilon < 1$. \ {The choice of $\epsilon$ will be specified later}.

\ {According to \eqref{EXnlPhiXnj}, we only need to estimate the terms $\mathbb{E} [(\bar {\bf X}^\omega_n)_l]^2$}. We first substitute the estimated result \eqref{amgmieq} into Eq.\eqref{sumEXBX}, and then substitute the estimated results of \eqref{sumEXBX} (including Eqns.\eqref{Esigmaxipsi}\eqref{1st}) into Eq.\eqref{EXnXn}. Combining all the estimate results for terms on the right hand side of Eq.\eqref{EXnXn}, we obtain an estimate for $\mathbb{E} [(\bar {\bf X}^\omega_n)_l]^2$ as follows, 
\begin{equation}
\mathbb{E} [(\bar {\bf X}^\omega_n)_l]^2 \le (\textbf{R}_n)_l + \epsilon\mathbb{E} [(\bar {\bf X}^\omega_n)_l]^2,
\label{EXLeREX}
\end{equation}
where $(\textbf{R}_n)_l$ denotes all the remaining terms with  $(\textbf{R}_n)_l=O(n)$. \ {Notice that $(\textbf{R}_n)_l$ also contains the term $(4\epsilon)^{-1}\mathbb{E}[(\bm\psi^{\omega(n\Delta t)}_{\Delta t}( {\bf X}^\omega_{n}))_l]^2$, which is $O(1)$ due to Eq.\eqref{EPhiXnEPhi} and the choice of $\epsilon$. We choose $0< \epsilon < 1$ (e.g. $\epsilon=1/3$), move the term $\epsilon \mathbb{E} [(\bar {\bf X}^\omega_n)_l]^2$ to the left hand side of \eqref{EXLeREX}, and divide both sides of the inequality by $(1-\epsilon)n$. We can obtain that $\frac{1}{n} \mathbb{E}||\bar {\bf X}_n^\omega||^2$ is bounded}. Therefore, we prove the claim in \eqref{3rdbound}.

Finally, we combine the estimate results in Eqns.\eqref{EXnXn}\eqref{1overnEBXBX}\eqref{sumEXBX}\eqref{1st}\eqref{EXnPhiXn} and obtain that 

\begin{align}
\frac{\mathbb{E}\bar {\bf X}^{\omega}_n\otimes \bar {\bf X}^\omega_n}{n\Delta t} = \sigma^2 {\bf I}_d +2 {\bf S}\mathbb{E} \bm \psi_{\Delta t}\otimes \tilde {\bf B}_{\Delta t}/\Delta t + O(\Delta t) + O(\frac{1}{\sqrt{n}\Delta t}).
\end{align}
According to Theorem \ref{cellconverge} and Remark \ref{order}, we have the estimate 

\ {
\begin{align}
\big|\big|2 {\bf S}\mathbb{E} \bm \psi_{\Delta t}\otimes \tilde {\bf B}_{\Delta t}/\Delta t - 2{\bf S}\mathbb{E}\bm \psi\otimes {\bf b}\big|\big|_{L^2(\mathcal{X})}= O(\Delta t^{\frac{2c_1}{2c_1+c_3}}) := \rho(\Delta t),
\end{align}
}
where $\lim_{\Delta t \to 0}\rho(\Delta t) = 0$.  Thus, the statement in \eqref{est:order-minusmean} is proved. 
\end{proof}
\begin{remark} 
Theorem \ref{thm:convergence} shows that when the time step $\Delta t$ is given and fixed, we have

\begin{align}\label{limerror}
\lim_{n\to\infty}\frac{\mathbb{E}\bar{\bf X}^{\omega}_n\otimes  \bar{\bf X}^\omega_n}{n\Delta t}= \sigma^2 {\bf I}_d +2 {\bf S}\int_{\mathcal{X}} \bm \psi\otimes {\bf b} dP_0 + \rho(\Delta t),
\end{align}
which reveals the connection of the definition of the effective diffusivity by solving discrete-type and continuous-type corrector problems. Our result appears to be the first one in the literature to build this connection. 
\end{remark}
Notice that in the Theorem \ref{thm:convergence}, we assume $\bar {\bf X}^{\omega}_n = {\bf X}^{\omega}_n - n \bar {\bf B}_{\Delta t}$ are given, where we use Monte Carlo method to compute $\bar {\bf B}_{\Delta t}$. In some cases, if we cannot calculate the drift constant $\bar {\bf B}_{\Delta t}$ exactly, we can directly estimate the term $\mathbb{E} {\bf X}^{\omega}_n\otimes  {\bf X}^\omega_n$, which is summarized in the following corollary.
\begin{corollary}\label{thm:convergencewithourdrift}
Let ${\bf X}^{\omega}_n$, $n=0,1,....$ be the numerical solution of the stochastic structure-preserving scheme \eqref{one-step} and $\Delta t$ be the time step that is fixed. Suppose $n (\Delta t)^3$ and $\frac{1}{\sqrt{n}\Delta t}$ are small enough, we have

\begin{align}\label{est:order}
	\frac{\mathbb{E} {\bf X}^{\omega}_n\otimes  {\bf X}^\omega_n}{n\Delta t}= \sigma^2 {\bf I}_d+2 {\bf S}\int_{\mathcal{X}} \bm\psi\otimes {\bf b} dP_0 + \rho(\Delta t) 
	+ O(\frac{1}{\sqrt{n}\Delta t}) + O\big(n(\Delta t)^3\big),
	\end{align}
where \ {$\rho(\Delta t)=O(\Delta t^{\frac{2c_1}{2c_1+c_3}})$} is a function satisfying $\lim_{\Delta t \to 0}\rho(\Delta t) = 0$ and is independent of 
the computational time $T$, and ${\bf S}$ represents the symmetrization operator.
\end{corollary}
\begin{proof}
Using the observation that 

\begin{align}\label{est:order2}
\frac{\mathbb{E} {\bf X}^{\omega}_n\otimes  {\bf X}^\omega_n}{n\Delta t} = \frac{\mathbb{E}\bar {\bf X}^{\omega}_n\otimes \bar {\bf X}^\omega_n}{n\Delta t} + \frac{2{\bf S}\mathbb{E}\bar {\bf X}^{\omega}_n \otimes \bar {\bf B}_{\Delta t}}{\Delta t} + \frac{n^2 \bar {\bf B}_{\Delta t}\otimes \bar {\bf B}_{\Delta t}}{n\Delta t}
\end{align}
and Theorem \ref{EBbound}, we can straightforwardly get the proof. 
\end{proof}

\begin{remark}\label{ProofBasedOnErgoditicy}
In our convergence analysis, we interpret the solution process generated by our numerical scheme as a Markov process. By exploring the ergodicity of the solution process (i.e., Markov process), we give a sharp error estimate of the proposed numerical scheme in computing effective diffusivity.  
\end{remark}

\begin{remark}\label{ConvergenceOrderDependonLemma41}
\ {
If ${\bf b}$ satisfies a H\"{o}lder-$\gamma$ continuous condition in the time domain with $0<\gamma<1$, we will obtain a weaker convergence rate in Lemma \ref{sec:Converge-discretocontinuous}, i.e. $||S_n f - S^{n\Delta t} f||_{L^2(\mathcal{X})} \le c_2 L (\Delta t)^{\gamma}$ with $0<\gamma<1$. Under such condition, we can still obtain convergence analysis of the numerical methods for computing effective diffusivity, e.g. Theorem \ref{thm:convergence} with a smaller convergence rate in $\rho(\Delta t) = O(\Delta t^{\frac{2\gamma c_1}{2c_1+c_3}})$.}
\end{remark}

\section{Numerical Results}\label{sec:numericaltest}
\noindent
The aim of this section is two-fold. First, we will verify the convergence results obtained in Section \ref{sec:converge-DE}. Second, we will use the proposed method to compute effective diffusivity in random flows, where incompressible random flows in two- and three-dimensional spaces will be studied. Without loss of generality, we compute the quantity $\frac{\mathbb{E}[( {\bf X}^{\omega}_{n,1})^2]}{2n\Delta t}$, which is used to approximate $D^{E}_{11}$ in the effective diffusivity matrix $D^{E}$. Notice that $  {\bf X}^{\omega}_{n,1}$ is the first component of the solution vector $  {\bf X}^{\omega}_{n}$. One can obtain $D^{E}_{11}$ by choosing ${\bf v}=(1,0)^{T}$ in the equation \eqref{continus-effective-diffusivity} of the Prop. \ref{continus-cellproblem}.  

\subsection{Numerical methods for generating random flows}
\noindent 
To start with, we discuss how to generate random flows that will be used in our numerical experiments. 
Assume the vector field ${\bf b}(t,{\bf x} ,\omega)$ has a spectral measure 

\begin{align}																							
\exp(-r({\bf k})|t|)\Gamma({\bf k})(\bf I - \frac{k\otimes k}{|k|^2}),
\label{spectralmeasure}
\end{align}
where ${\bf k}=(k_1,k_2)^{T}$ or ${\bf k}=(k_1,k_2,k_3)^{T}$, $r({\bf k}) > c_0$ for some positive constant $c_0$, and $\Gamma({\bf k})$ is integrable and decays fast for large $\bf k$. Under such settings, the velocity field ${\bf b}(t,{\bf x},\omega)$ satisfies the $\rho$ mixing condition and is stationary and divergence-free \cite{rosenblatt2012markov,doukhan2012mixing}. In order to mimic the energy spectrum of real flows, we assume $\Gamma({\bf k}) \propto 1/|{\bf k}|^{2\alpha+d-2}$ with ultraviolet cutoff $|{\bf k}| \le K < \infty$ and $r({\bf k}) \propto |{\bf k}|^{2\beta}$. 
The spectral gap condition \ref{TimeRelaxation} requires $\beta \le 0$ and the integrability of $\Gamma({\bf k})$ requires $\alpha < 1$. Here for simplicity, we choose $\beta = 0$. 

Given the spectral measure \eqref{spectralmeasure}, we use the randomization method \cite{kraichnan1970diffusion,Majda:99} to generate realizations of the velocity field. Specifically, we 
approximate it as

\begin{align}
{\bf b}(t,{\bf x}) = \frac{1}{\sqrt{M}}\sum_{m=1}^{M} \big[{\bf u}_m \cos({\bf k}_m\cdot {\bf x} ) + {\bf v}_m \sin ({\bf k}_m \cdot {\bf x} )\big].
\label{randomization-method-v}
\end{align}
Notice that we have suppressed the dependence of the velocity on $\omega$ for notation simplicity here. 
In fact, the parameters ${\bf k}_m$, ${\bf u}_m$ and ${\bf v}_m$ contain randomness. The spectrum points ${\bf k}_m$ were chosen independently according to the spectral measure $\Gamma(\bf k)$. Due to the isotropicity, we first generate a point uniformly distributed on the unit sphere or unit circle, which represents the direction of the ${\bf k}_m$. Then we generate the length $r$ of ${\bf k}_m$, which satisfies a density function $\rho(r) \propto 1/r^{2\alpha - 1}, 0 < r \le K$.

For the random flows in two-dimensional space, we have 

\begin{align}
{\bf u}_m = \xi_m(t)\frac{{\bf k}^\perp_m}{|{\bf k}^\perp_m|}, \quad {\bf v}_m = \eta_m(t)\frac{{\bf k}^\perp_m}{|{\bf k}^\perp_m|}, \quad {\bf k}_m=(k_m^1,k_m^2), \quad m=1,...,M,
\label{2D-setting}
\end{align}
where ${\bf k}_m^\perp = (-k_m^2,k_m^1)$, $\xi_m(t)$ and $\eta_m(t)$ are independent 1D  Ornstein-Uhlenbeck (OU) processes with covariance function 

\begin{align*}
Cov(\xi_m(t_1),\xi_m(t_2)) = Cov(\eta_m(t_1),\eta_m(t_2)) =\exp(-\theta |t_1 - t_2|).
\end{align*}
Here $\theta>0$ is a parameter to control the roughness of the OU process.  
To obtain the OU path for $\xi_m(t)$, we generate a series of $\{\xi_m(n\Delta t)\}$ satisfies 

\begin{align}
\xi_m(n\Delta t) = e^{-\theta\Delta t}\xi_m((n-1)\Delta t) + \sqrt{1-e^{-2\theta\Delta t}} \zeta_{m,n},\quad n=1,2,3,...
\label{numerical-OU-path}  
\end{align}
where $\xi_m(0)$, $\zeta_{m,n}$, $m,n=1,2,3,...$ are i.i.d. $N(0,1)$ distributed random variables. One can easily verify that $Cov(\xi_m(i\Delta t),\xi_m(j\Delta t)) = \exp(-\theta |i-j|\Delta t)$. The OU path for $\eta_m(t)$ can be generated by using the same approach. 

For the random flows in three-dimensional space, we have 

\begin{align}
{\bf u}_m = {\bm \xi}_m(t) \times \frac{{\bf k}_m}{|{\bf k}_m|}, \quad {\bf v}_m = {\bm \eta}_m(t) \times \frac{{\bf k}_m}{|{\bf k}_m|}, \quad  {\bf k}_m=(k_m^1,k_m^2,k_m^3),
\label{3D-setting}
\end{align} 
where the samples ${\bm \xi}_m(t)$ and ${\bm \eta}_m(t)$ are independent 3D random vectors, whose components are independent stationary OU process having the covariance function $Cov(\bm\xi_m(t_1),\bm\xi_m(t_2)) = Cov(\bm\eta_m(t_1),\bm\eta_m(t_2)) =\exp(-\theta |t_1 - t_2|){\bf I}_{3}$.  Each component of ${\bm \xi}_m(t)$ and ${\bm \eta}_m(t)$ can be generated by using the method 
\eqref{numerical-OU-path}. \ {One can easily verify that the random velocity fields 
generated by Eq.\eqref{randomization-method-v} with the setting \eqref{2D-setting} in two-dimensional space and \eqref{3D-setting} in three-dimensional space automatically satisfy the divergence-free condition}. 
 
\subsection{Verification of the convergence analysis}\label{sec:verifyconvergence}
\noindent 
In this subsection, we study the convergence rate of our method in computing incompressible random flow in 2D and 3D spaces. 

For the random flow in 2D space, we solve the SDE \eqref{eqn:generalSDEDefD}, where the velocity filed is chosen as \eqref{randomization-method-v} with the setting \eqref{2D-setting}. The velocity field were simulated with $M = 1000$. The parameters in the spectral measure $\Gamma(\bf k)$ are $K = 10$ and $\alpha = 0.75$. The time-mixing constant $\theta = 10$ in the covariance function.  The molecular diffusivity $\sigma = 0.1$. We use Monte Carlo method to generate independent samples for the Brownian motion ${\bf w}(t)$ and velocity field ${\bf b}(t,{\bf x})$.  The sample number is denoted by $N_{mc}$. 

We choose time step $\Delta t_{ref}=0.001$ and $N_{mc}=100,000$ to solve the SDE \eqref{eqn:generalSDEDefD} and compute the reference solution, i.e., the ``exact'' effective diffusivity, where the final computational time is $T=22$ so that the calculated effective diffusivity converges to a constant. It takes about 24 hours to compute the reference solution on a 64-core server (Gridpoint System at HKU). The reference solution for the effective diffusivity is 
$D^{E}_{11}=0.1736$.

For the random flow in 3D space, we solve the SDE \eqref{eqn:generalSDEDefD}, where the velocity field is chosen as \eqref{randomization-method-v} with the setting \eqref{3D-setting}. The velocity field were simulated with $M = 100$. The parameters in the spectral measure $\Gamma(\bf k)$ are $K = 10$ and $\alpha = 0.75$. The time-mixing constant $\theta = 10$ in the covariance function. The molecular diffusivity $\sigma = 0.1$. Again, we use Monte Carlo method to generate dependent samples for the Brownian motion ${\bf w}(t)$ and velocity field ${\bf b}(t,{\bf x})$.  

We choose $\Delta t_{ref}=0.001$ and $N_{mc}=180,000$ to solve the SDE \eqref{eqn:generalSDEDefD} and compute the reference solution, i.e., the ``exact'' effective diffusivity, where the final computational time is $T=25$ so that the calculated effective diffusivity converges to a constant. It takes about 21 hours to compute the reference solution on a 64-core server (Gridpoint System at HKU). The reference solution for the effective diffusivity is $D^{E}_{11}=0.1137$. 
We remark that in our numerical experiment, we choose $M=1000$ for 2D random flow and $M=100$ for 3D random flow so that the velocity field numerically satisfies the ergodicity assumption.  

In Fig.\ref{convergece-dt-2D}, we plot the convergence results of the 
effective diffusivity for the 2D random flow using our method (i.e., $\frac{\mathbb{E}[(  {\bf X}^{\omega}_{n,1})^2]}{2n\Delta t}$) with respective to different time-step $\Delta t$ at $T=22$, where the number of the Monte Carlo samples $N_{mc}=50,000$. In addition, we show a fitted straight line with the slope $1.17$, i.e., the convergence rate is about $O(\Delta t)^{1.17}$. Similarly, we show the convergence results of $\frac{\mathbb{E}[(  {\bf X}^{\omega}_{n,1})^2]}{2n\Delta t}$ for the 3D random flow in Fig.\ref{convergece-dt-3D} with respective to different time-step $\Delta t$ at $T=25$, where the number of the Monte Carlo samples $N_{mc}=50,000$. We also show 
a fitted straight line with the slope $0.98$, i.e., the convergence rate is about $O(\Delta t)^{0.98}$. These numerical results agree with our error analysis. 

\begin{figure}[h]
	\centering
	\begin{subfigure}{0.49\textwidth}
		\includegraphics[width = 21em]{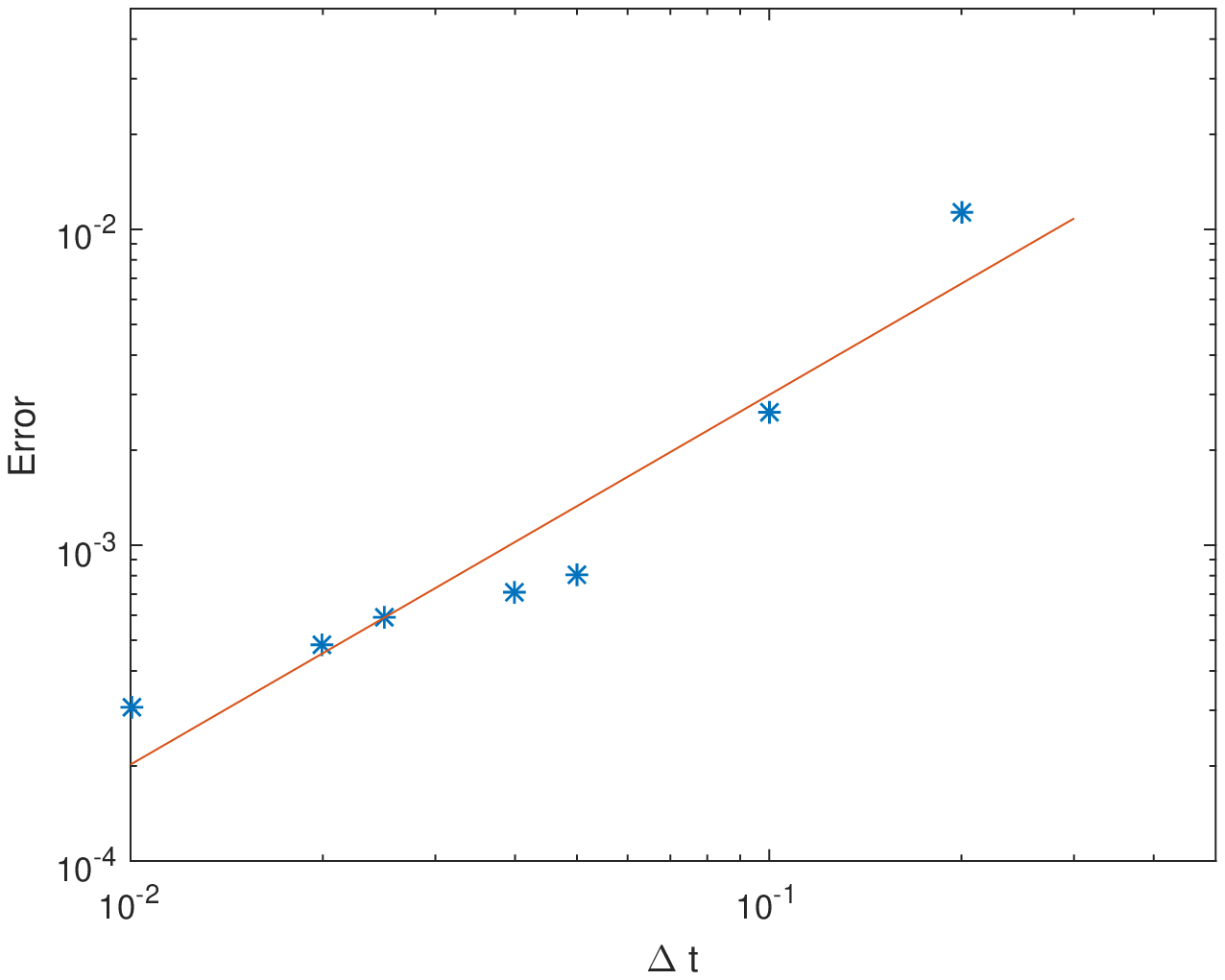}
		\caption{2D random flow, fitted slope $\approx$ 1.17}
		\label{convergece-dt-2D} 
	\end{subfigure}
	\begin{subfigure}{0.49\textwidth}
		\includegraphics[width = 21em]{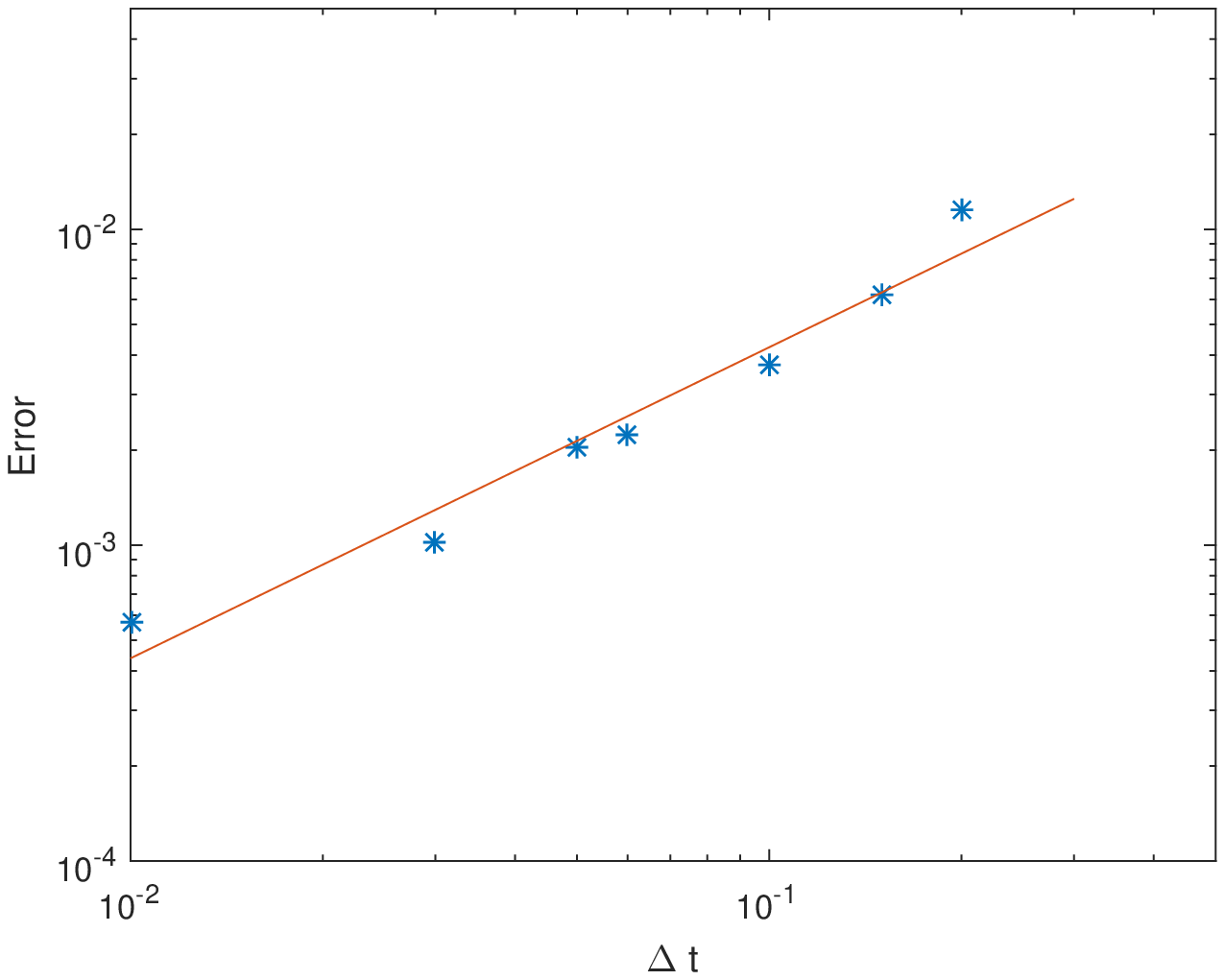}
		\caption{3D random flow, fitted slope $\approx$ 0.98.}
		\label{convergece-dt-3D}
	\end{subfigure}
	\caption{Error of $D^{E}_{11}$ for random flows with different time-steps.}
	\label{convergece-dt}
\end{figure}

\subsection{Comparasion between the volume-preserving scheme and Euler scheme}\label{sec:Compare}
\noindent\ {
To demonstrate the benefit of our method in computing effective diffusivity, we compare the performance of 
the volume-preserving scheme and Euler-Maruyama scheme (also called the Euler scheme).} 

\ {
For the random flow in 2D space, we solve the SDE \eqref{eqn:generalSDEDefD}, where the velocity filed is chosen as \eqref{randomization-method-v} with the setting \eqref{2D-setting}. The time-mixing constant $\theta = 1$ in the covariance function and other parameters are the same as that were used in Section \ref{sec:verifyconvergence}. We use the volume-preserving scheme with $\Delta t_{ref}=0.003125$ and $N_{mc}=100,000$ to solve the SDE \eqref{eqn:generalSDEDefD} and compute the reference solution, i.e., the ``exact'' effective diffusivity, where the final computational time is $T=54$ so that the calculated effective diffusivity converges to a constant. It takes about 24 hours to compute the reference solution on a 64-core server (Gridpoint System at HKU). The reference solution for the effective diffusivity is $D^{E}_{11}=0.3610$.}

\ { 
In Fig.\ref{convergececomp-dt-2D},  we plot the convergence results of the effective diffusivity for the 2D random flow using the volume-preserving scheme and the Euler scheme with respect to different time-step $\Delta t$ at $T=54$, where the number of the Monte Carlo samples $N_{mc}=50,000$. The slopes of the fitted lines for the volume-preserving scheme and the Euler scheme are $0.86$ and $0.62$, respectively. In addition, we can see that 
the volume-preserving scheme reduces the numerical error by more than one order of magnitude than that of the  Euler scheme by using the same time-step  $\Delta t$.
}

\ {
For the random flow in 3D space, we solve the SDE \eqref{eqn:generalSDEDefD}, where the velocity field is chosen as \eqref{randomization-method-v} with the setting \eqref{3D-setting}. The time-mixing constant $\theta = 4$ in the covariance function and other parameters are the same as that were used in Section \ref{sec:verifyconvergence}. We use the volume-preserving scheme with $\Delta t_{ref}=0.003125$ and $N_{mc}=100,000$ to solve the SDE \eqref{eqn:generalSDEDefD} and compute the reference solution, where the final computational time is $T=40$ so that the calculated effective diffusivity converges to a constant. It takes about 32 hours to compute the reference solution on a 64-core server (Gridpoint System at HKU). The reference solution for the effective diffusivity is $D^{E}_{11}=0.2266$.}

\ { 
In Fig.\ref{convergececomp-dt-3D},  we plot the convergence results of the effective diffusivity for the 3D random flow using the volume-preserving scheme and the Euler scheme with respect to different time-step $\Delta t$ at $T=40$, where the number of the Monte Carlo samples $N_{mc}=50,000$. The slopes of the fitted lines for the volume-preserving scheme and the Euler scheme are $0.91$ and $0.44$, respectively. Again, we can see that the volume-preserving scheme significantly reduces the numerical error by more than one order of magnitude than that of the  Euler scheme by using the same time-step  $\Delta t$.	}

\ {
We remark that the volume-preserving scheme is an implicit scheme which needs to use Newton's iteration method to solve the corresponding nonlinear equations. In our numerical experiments, we use the numerical solutions at time $t=t_n$ as an initial guess for the solution at time $t=t_{n+1}$. We find this approach is very efficient, i.e., three or four steps of iterations will give convergent results.  Thus, the computational cost for the volume-preserving scheme is about three or four times of Euler scheme in the same setting. However, the volume-preserving scheme is superior to the Euler scheme due to its faster convergence rate and smaller magnitude in the numerical error. }
	\begin{figure}[h]
		\centering
		\begin{subfigure}{0.49\textwidth}
			\includegraphics[width = 21em]{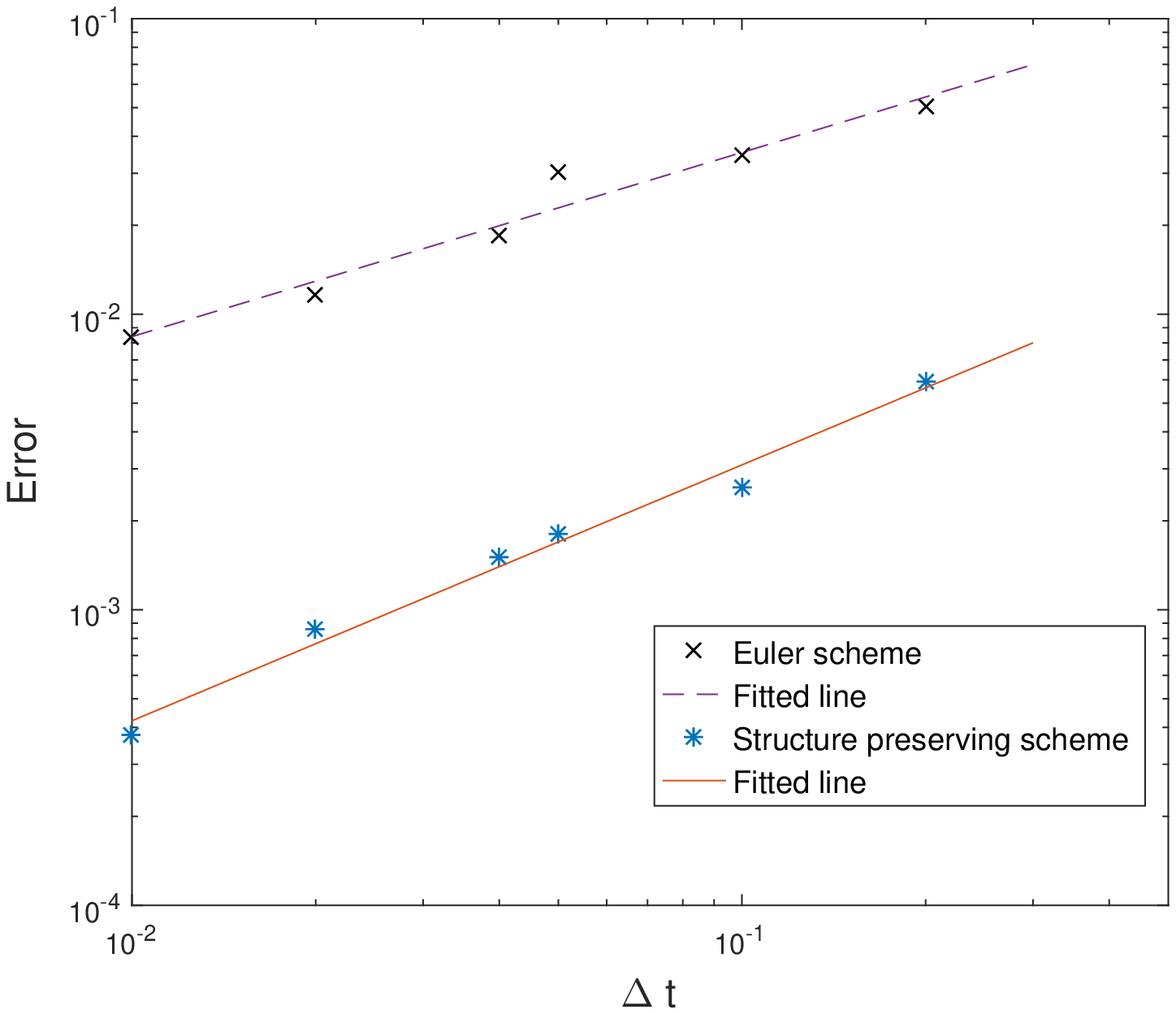}
			\caption{2D comparasion: fitted slope for Euler scheme is 0.62, fitted slope for volume preserving scheme is 0.86.}
			\label{convergececomp-dt-2D} 
		\end{subfigure}
		\begin{subfigure}{0.49\textwidth}
			\includegraphics[width = 21em]{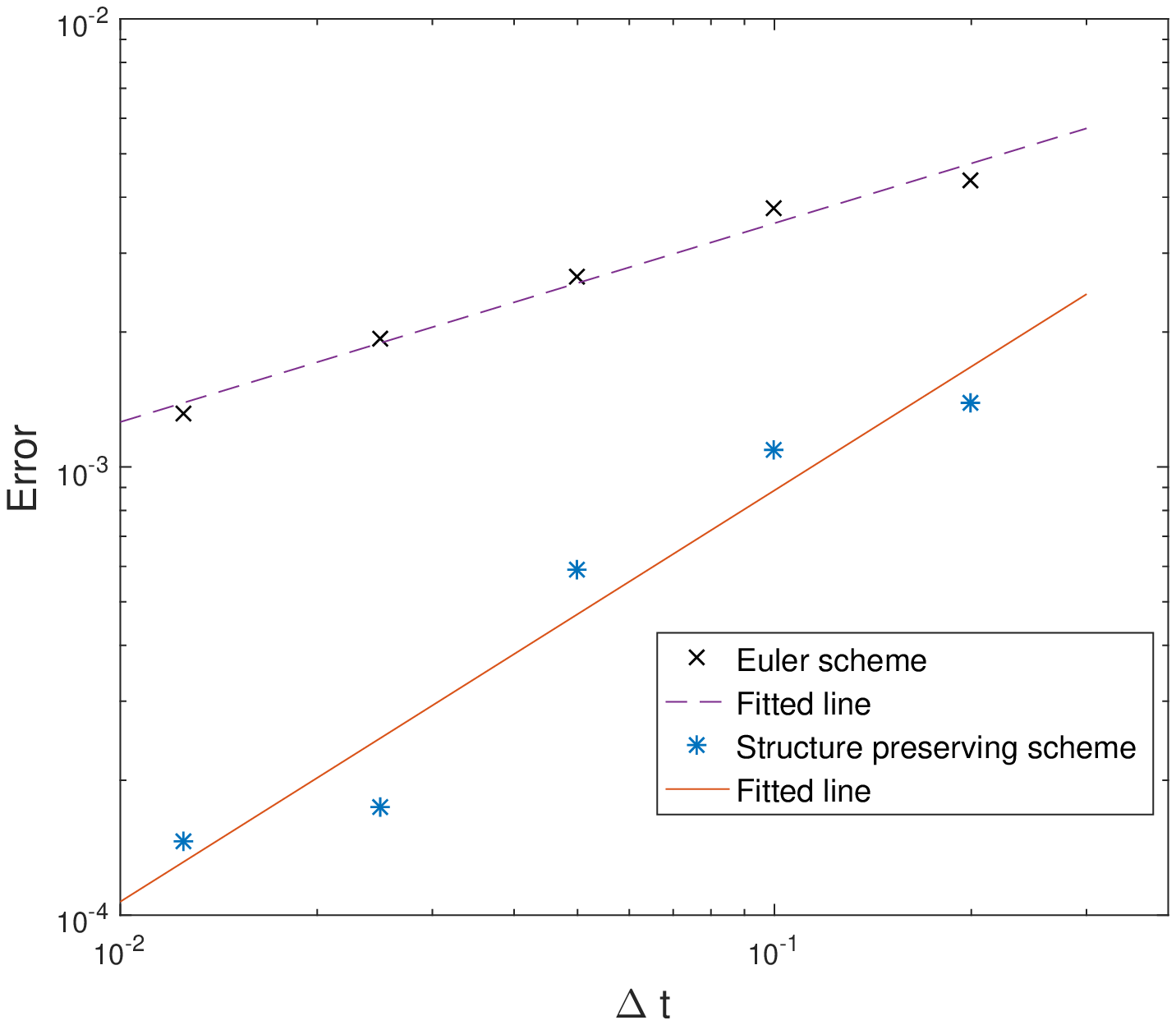}			
			\caption{3D comparasion: fitted slope for Euler scheme is 0.44, fitted slope for volume preserving scheme is 0.91.}
			\label{convergececomp-dt-3D}
		\end{subfigure}
		\caption{Comparasion between Euler scheme and volume preserving scheme}
		\label{convergececompare}
	\end{figure}

\subsection{Verification of the exponential decay property.}\label{sec:verifyspectral-gap}
\noindent 
The time relaxation property \eqref{TimeRelaxation}, which is equivalent to the exponential decay property \eqref{exponentialdecayproperty}, plays an important role in the existence of the effective diffusivity; see Prop. \ref{continus-cellproblem}. In Theorem \ref{thm:exponentialdecay}, we prove that the numerical solutions inherit  the exponential decay property. Based on this key fact, we can define the discrete-type corrector problem and prove the convergence analysis of our method. In this subsection, we will verify that the velocity field propagated by the random flow \eqref{randomization-method-v} has the exponential decay property, where both the 2D and 3D cases will be tested.  

In the experiment for 3D random flow, we choose the time step size $\Delta t = 0.05$. The velocity field 
will be approximated by $M = 100$ terms in \eqref{randomization-method-v} with the setting \eqref{3D-setting}. The parameters in the spectral measure $\Gamma(\bf k)$ are $K = 10$ and $\alpha = 0.75$. The molecular diffusivity $\sigma = 0.1$. We randomly generate 200 samples $\{{\bf k}_m^{i}, {\bm \xi}_m^{i}(0),{\bm \eta}_m^{i}(0),m=1,...,M\}$, $i=1,...,200$, which will be used to generate initial states for the velocity field \eqref{randomization-method-v}, i.e., 

\begin{align*}
{\bf b}^i(0,{\bf x}) = \frac{1}{\sqrt{M}}\sum_{m=1}^{M} \big[{\bm \xi}_m^{i}(0)\times \frac{{\bf k}^i_m}{|{\bf k}^i_m|} \cos({\bf k}_m^i\cdot {\bf x} ) + {\bm \eta}^i_m(0)\times \frac{{\bf k}^i_m}{|{\bf k}^i_m|} \sin ({\bf k}_m^i \cdot {\bf x} )\big], \quad i=1,...,200. 
\end{align*}
Then, for each initial state ${\bf b}^i(0,{\bf x})$, we generate $5000$ different samples of the OU paths ${\bm \xi}^{i,p}_m(n\Delta t)$ and ${\bm \eta}^{i,p}_m(n\Delta t)$ and
Brownian motion paths $\bm{w}^{i,p}(n\Delta t)$, $1\le p\le 5000$. Given the sample data, we calculate the corresponding solution paths $\{{\bf X}_n^{i,p}\}_{0\leq n <\infty}$ and then calculate the value 

\begin{align}
{\bf b}^{i,p}(n\Delta t,{\bf X}^{i,p}_n) &=  \frac{1}{\sqrt{M}}\sum_{m=1}^{M} [{\bm \xi}_m^{i,p}(n\Delta t)\times \frac{{\bf k}^i_m}{|{\bf k}^i_m|} \cos({\bf k}_m^i\cdot {\bf X}^{i,p}_n) + {\bm \eta}^i_m(n\Delta t)\times \frac{{\bf k}^i_m}{|{\bf k}^i_m|} \sin ({\bf k}_m^i \cdot {\bf X}^{i,p}_n )],
\nonumber \\
& i=1,...,200, \quad 1\le p\le 5000.
\end{align}
Finally, we compute $\bar {\bf b}_n^i = \frac{1}{5000}\sum_{p=1}^{5000} {\bf b}^{i,p}(n\Delta t,{\bf X}^{i,p}_n)$ and the sample variance of $\bar {\bf b}_n^i$ with respect to $i$. This is an approximation to the value $||S_n {\bf b}||_{L^2(\mathcal{X})}$, which should satisfies exponential-decay property according to our analysis. The experiment for 
2D random flow is almost the same except the setting of the velocity filed \eqref{3D-setting} is replaced by \eqref{2D-setting} and we choose $M=1000$.
  

In Fig.\ref{verify-spectral-gap-2D} and Fig.\ref{verify-spectral-gap-3D}, we plot the calculated sample variance of the first component of $\bar {\bf b}_n^i$ for the 2D random flow and 3D random flow, respectively. We observe exponential decay of the sample variance with respect to time. Moreover, we find that larger $\theta$ leads to a faster decay in the sample variance, since larger $\theta$ results in a faster decorrelation in the random flow. Our numerical results show that the exponential decay property (see Theorem \ref{thm:exponentialdecay}) holds for the random flows we studied here. 


\begin{figure}[h]
	\centering
	\begin{subfigure}{0.49\textwidth}
		\includegraphics[width = 21em]{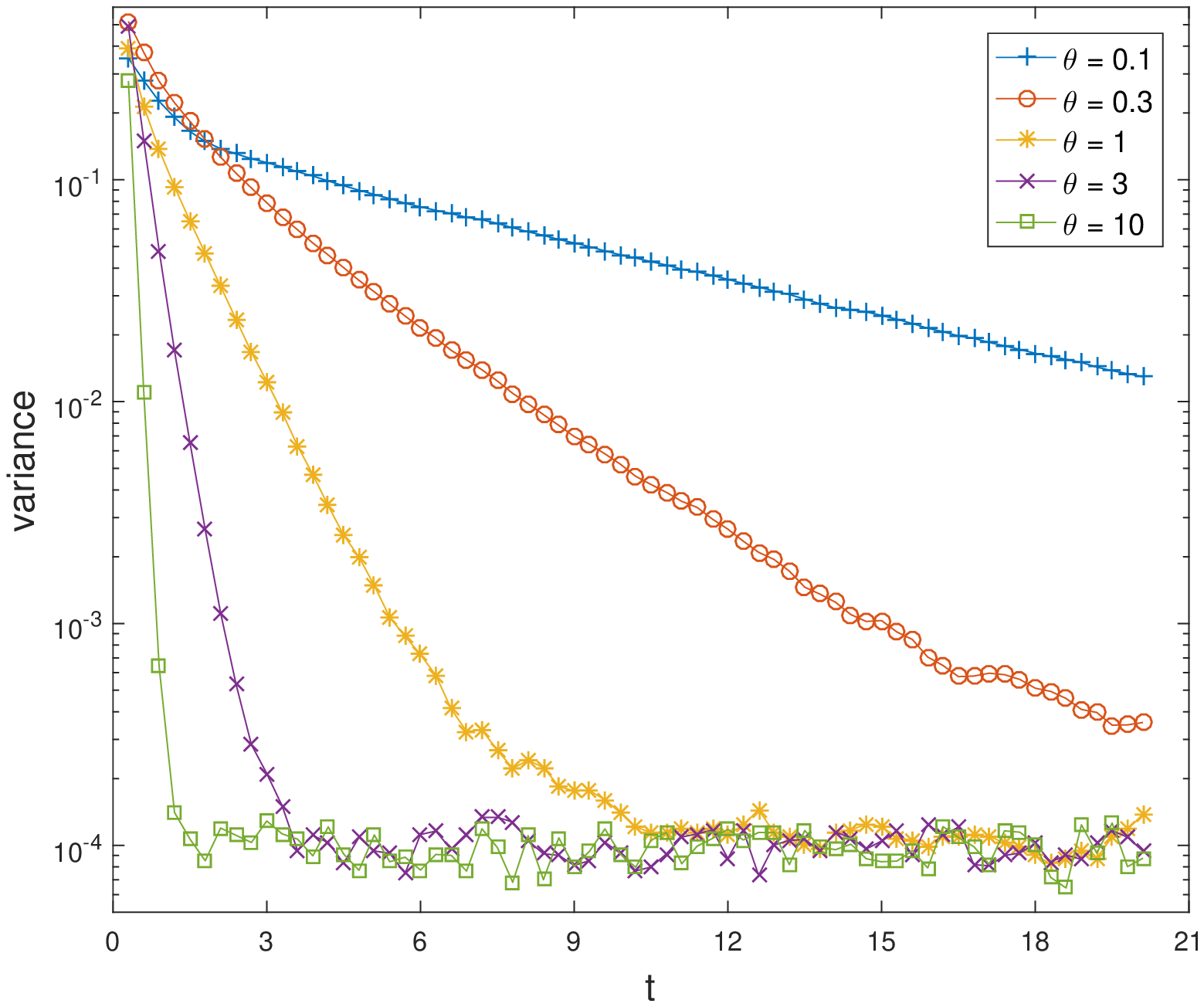}
		\caption{Calculated variance in the 2D flow along time.}
		\label{verify-spectral-gap-2D} 
	\end{subfigure}
	\begin{subfigure}{0.49\textwidth}
		\includegraphics[width = 21em]{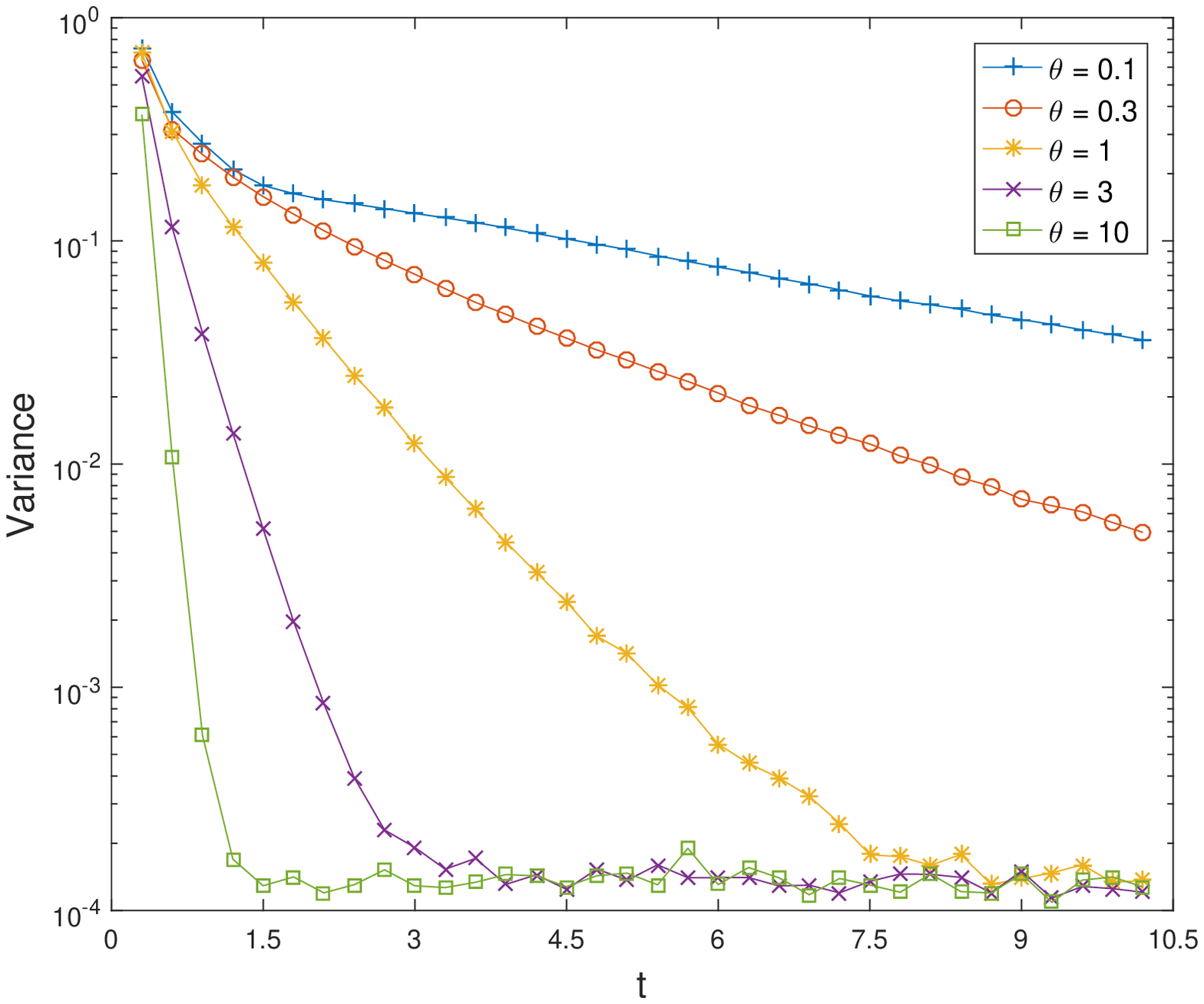}
		\caption{Calculated variance in the 3D flow along time.}
		\label{verify-spectral-gap-3D}
	\end{subfigure}
	\caption{Decay behaviors of the sample variance in 2D and 3D random flows.}
	\label{verify-spectral-gap}
\end{figure}

\subsection{Investigation of the convection-enhanced diffusion phenomenon}\label{sec:convection-enhanced}
\noindent 
In the first experiment, we study the relation between the numerical effective diffusivity $\frac{E[(  {\bf X}^{\omega}_{n,1})^2]}{2n\Delta t}$ and the parameter $\theta$, which controls the de-correlation rate in the  temporal dimension of the random flow.  In this experiment, the setting of the velocity field and the 
implementation of our method is the same as we used in Section \ref{sec:verifyspectral-gap}. We only 
choose different parameter $\theta$ to compute the numerical effective diffusivity.

In Fig.\ref{DE11-vs-theta-2D}, we plot the numerical effective diffusivity of 2D random flow obtained at different computational times, where the flow is generated with different $\theta$. The result for 3D random flow is shown in Fig.\ref{DE11-vs-theta-3D}. We find that different $\theta$ affects the mixing time of the system. When we increase the $\theta$, the system will quickly enter a mixing stage.  

\begin{figure}[h]
	\centering
	\begin{subfigure}{0.49\textwidth}
		\includegraphics[width = 21em]{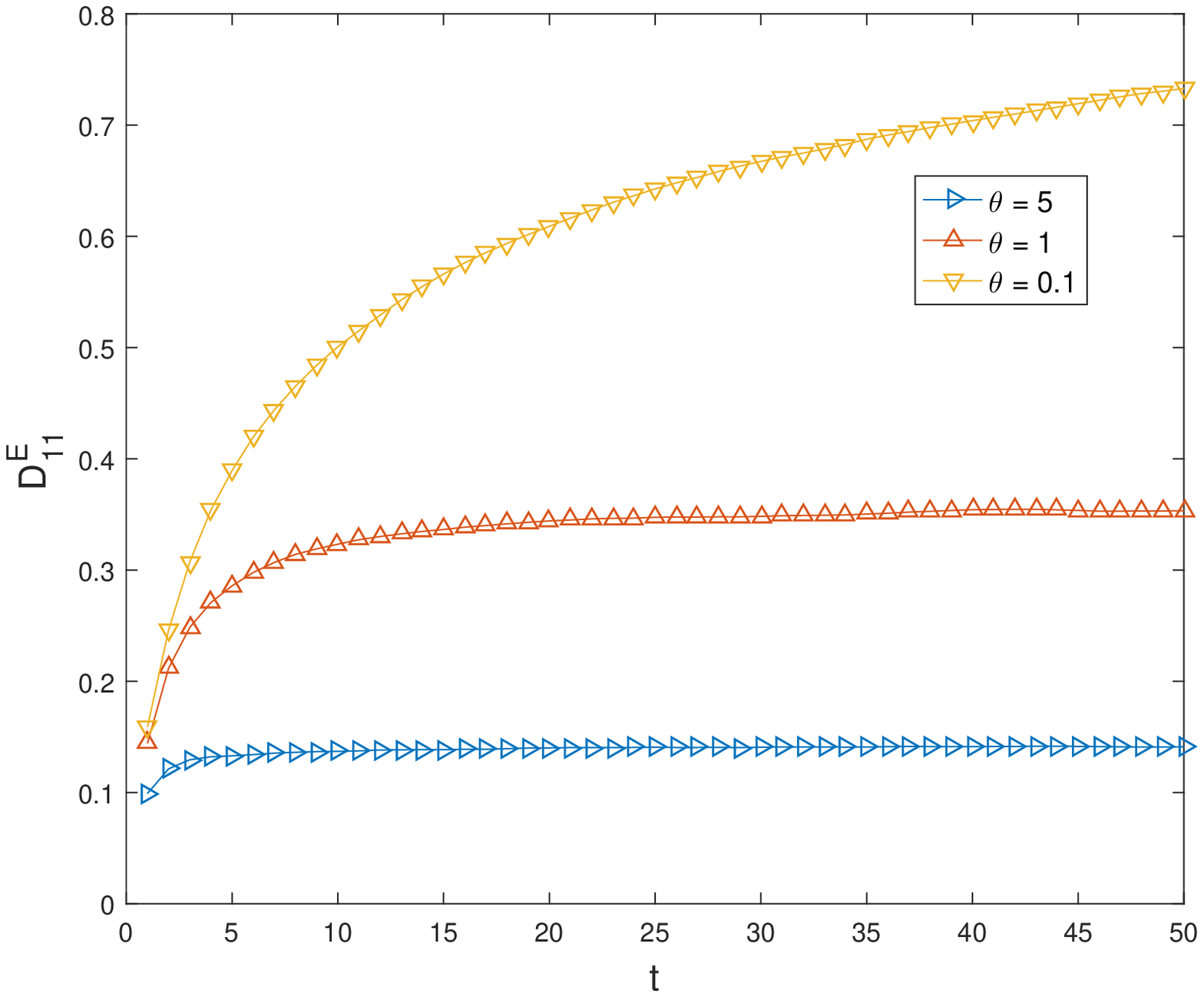}
		\caption{The quantity $\frac{E[(  {\bf X}^{\omega}_{n,1})^2]}{2n\Delta t}$ in the 2D flow along time.}
		\label{DE11-vs-theta-2D} 
	\end{subfigure}
	\begin{subfigure}{0.49\textwidth}
		\includegraphics[width = 21em]{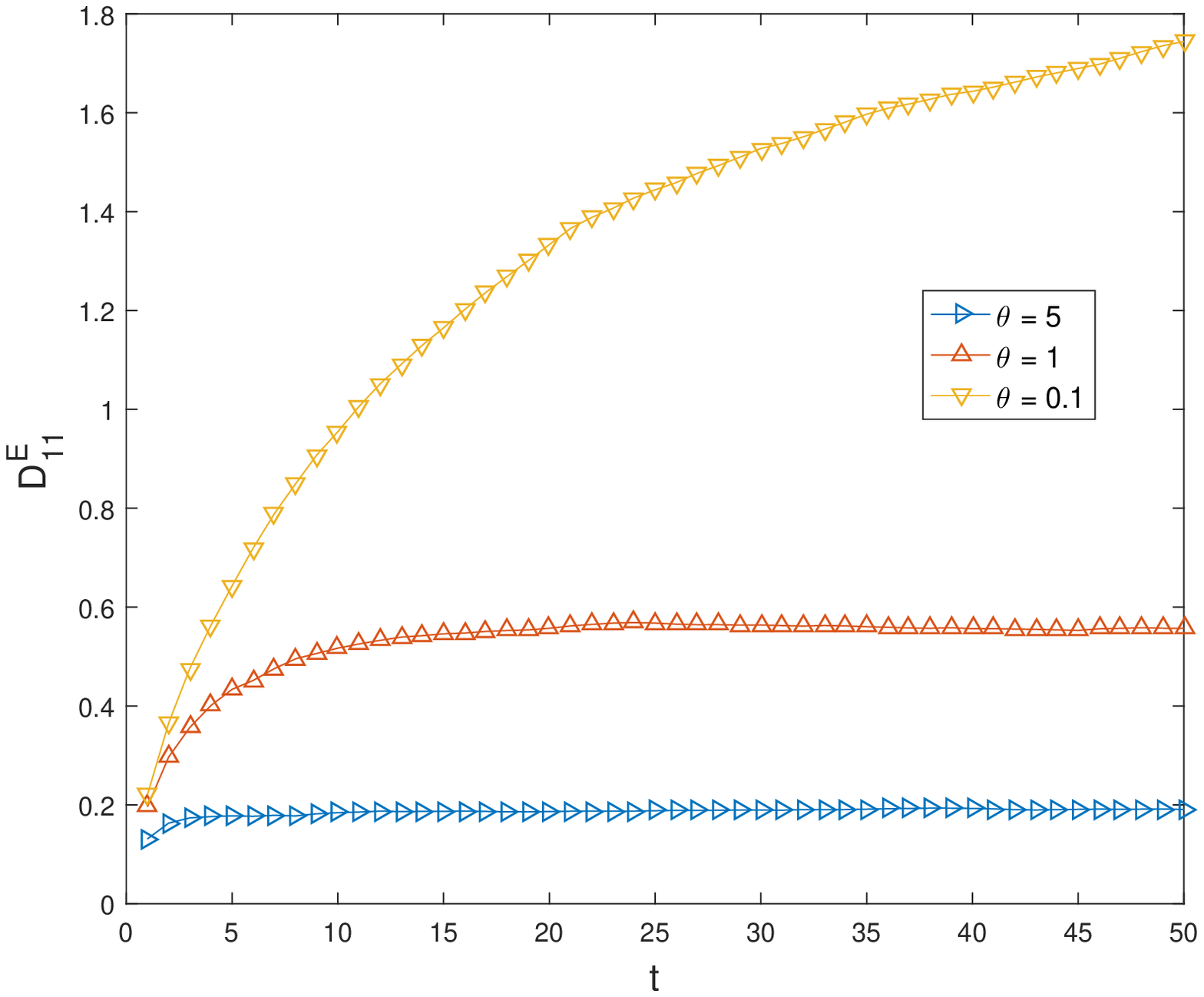}
        \caption{The quantity $\frac{E[(  {\bf X}^{\omega}_{n,1})^2]}{2n\Delta t}$ in the 3D flow along time.}
        \label{DE11-vs-theta-3D} 
	\end{subfigure}
	\caption{The relation between numerical effective diffusivity and $\theta$.}
	\label{DE11-vs-theta}
\end{figure}

In the second experiment, we choose different molecular diffusivity $\sigma$ to compute the corresponding numerical effective diffusivity, which allows us to study the existence of residual diffusivity for this random flow.  The residual diffusivity, a special yet remarkable convection-enhanced diffusion phenomenon, refers to the non-zero and finite effective diffusivity in the limit of zero molecular diffusivity as a result of a fully chaotic mixing of the streamlines. 

In the experiment for 2D random flow, we choose the time step $\Delta t = 0.05$, the velocity field were simulated with $M = 1000$, the time-mixing constant $\theta=0.1$ and the parameters in the spectral measure $\Gamma(\bf k)$ are $K = 10$ and $\alpha = 0.75$. For the 3D random flow, we choose $M=100$ and keep other parameters the same.  

Let $\kappa = \sigma^2/2$. In Fig.\ref{DE11-vs-sigma-2D}, we show the relation between numerical effective diffusivity of 2D random flow obtained at different computational times, where the result is generated with different $\sigma$.  The result for 3D random flow is shown in 
Fig.\ref{DE11-vs-sigma-3D}. We find that as $\kappa$ approaches zero, the quantity $\frac{\mathbb{E}[(\bar {\bf X}^{\omega}_{n,1})^2]}{2n\Delta t}$ converges to a non-zero (positive) constant, which indicates the existence of residual diffusivity in the random flows here. 
 
\begin{figure}[htbp]
	\centering
	\begin{subfigure}{0.49\textwidth}
		\centering
		\includegraphics[width = 21em]{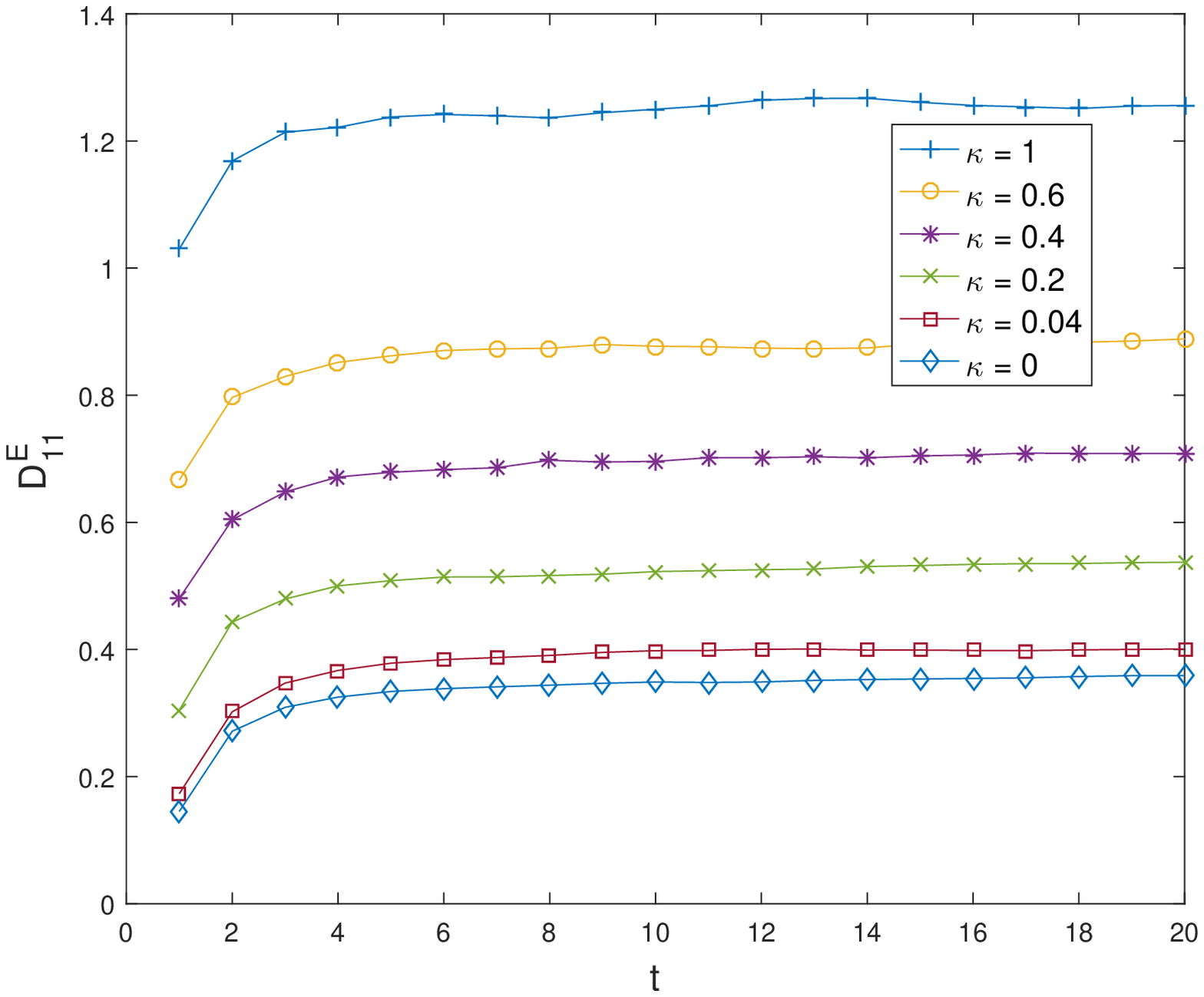}
		\caption{The quantity $\frac{\mathbb{E}[(  {\bf X}^{\omega}_{n,1})^2]}{2n\Delta t}$ 
			in the 2D flow along time.}
		\label{DE11-vs-sigma-2D}  
	\end{subfigure}
	\begin{subfigure}{0.49\textwidth}
		\centering
		\includegraphics[width = 21em]{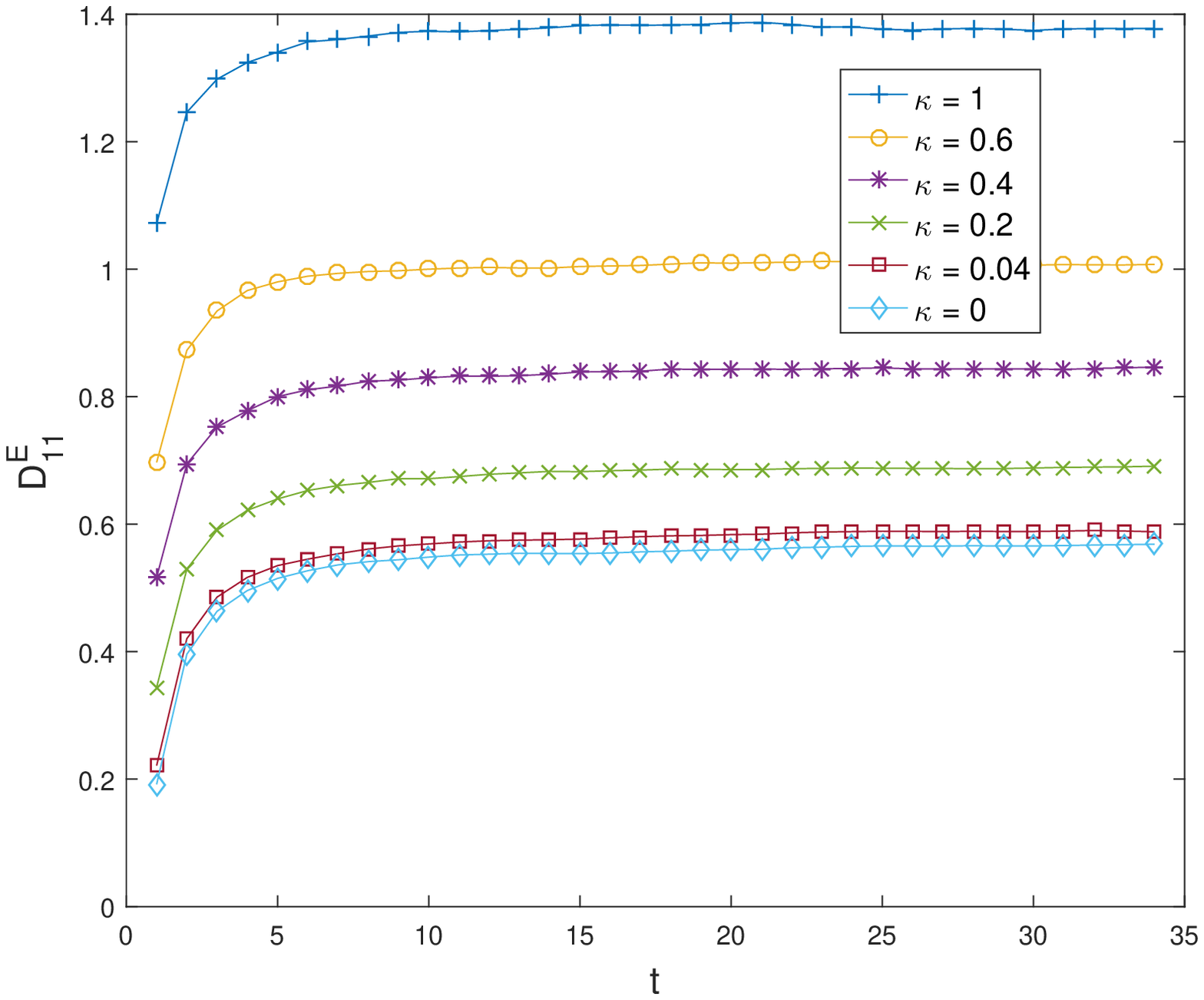}
		\caption{The quantity $\frac{\mathbb{E}[(  {\bf X}^{\omega}_{n,1})^2]}{2n\Delta t}$ 
			in the 3D flow along time.}
		\label{DE11-vs-sigma-3D}  
	\end{subfigure}
	\caption{The relation between numerical effective diffusivity and $\kappa = \sigma^2/2$, where 
		$\sigma$ is molecular diffusivity.}
	\label{DE11-vs-sigma}
\end{figure} 

In Fig.\ref{DE-to-D0-2D} and Fig.\ref{DE-to-D0-3D}, we plot the convergence behaviors of $D^E_{11}(\kappa)$ approaching $D^E_{11}(0)$ for the 2D and 3D random flows, respectively, when the systems enter a mixing stage.  The convergence behaviors when $\kappa$ approaches zero are slightly different though, both figures show that residual diffusivity exists in the random flows we studied here.

\begin{figure}[h]
	\centering
	\begin{subfigure}{0.49\textwidth}
		\centering
		\includegraphics[width = 21em]{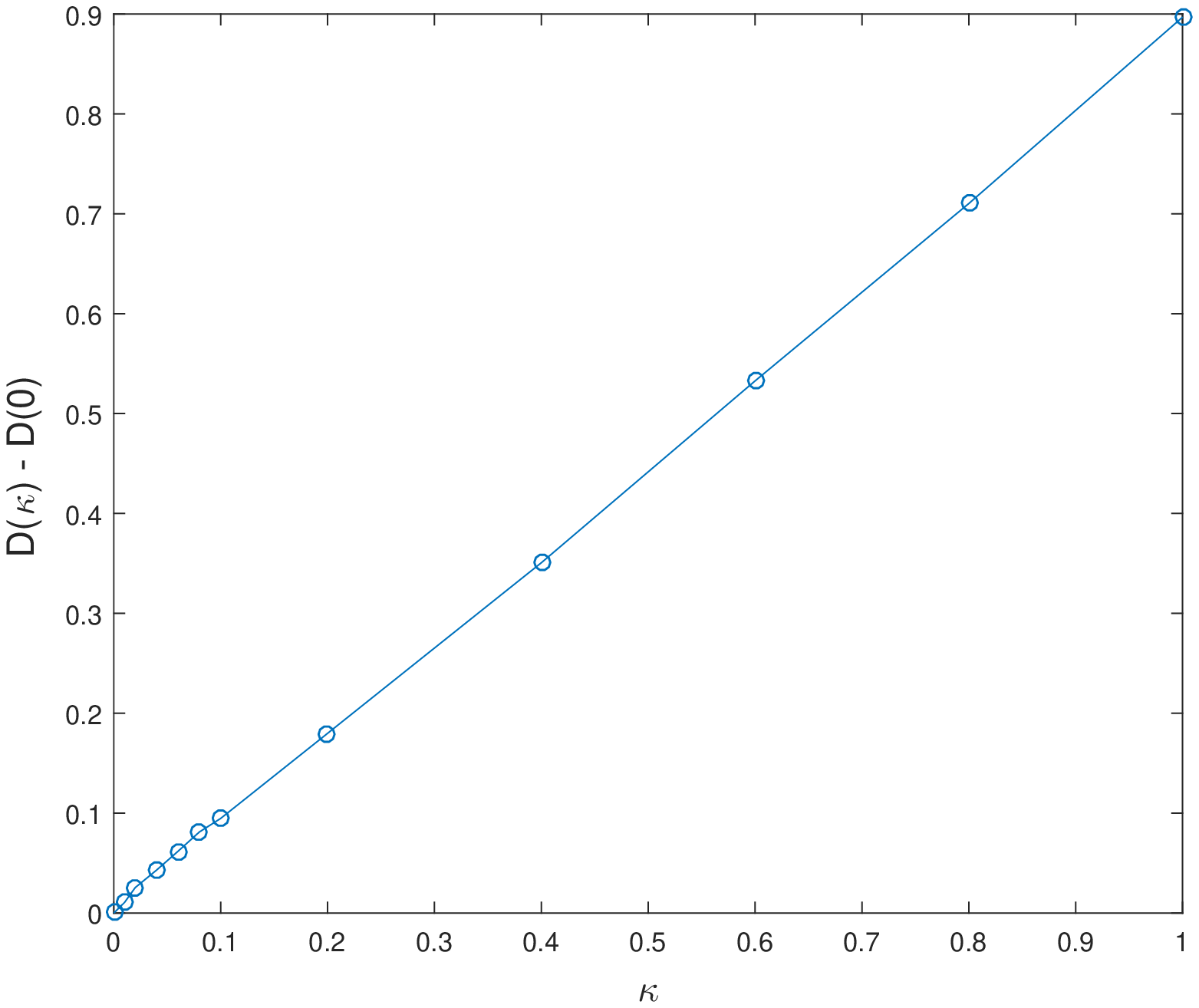}
		\caption{Results for the 2D random flow, $D_0 \approx 0.3584$.}
		\label{DE-to-D0-2D}  
	\end{subfigure}
	\begin{subfigure}{0.49\textwidth}
		\centering
		\includegraphics[width = 21em]{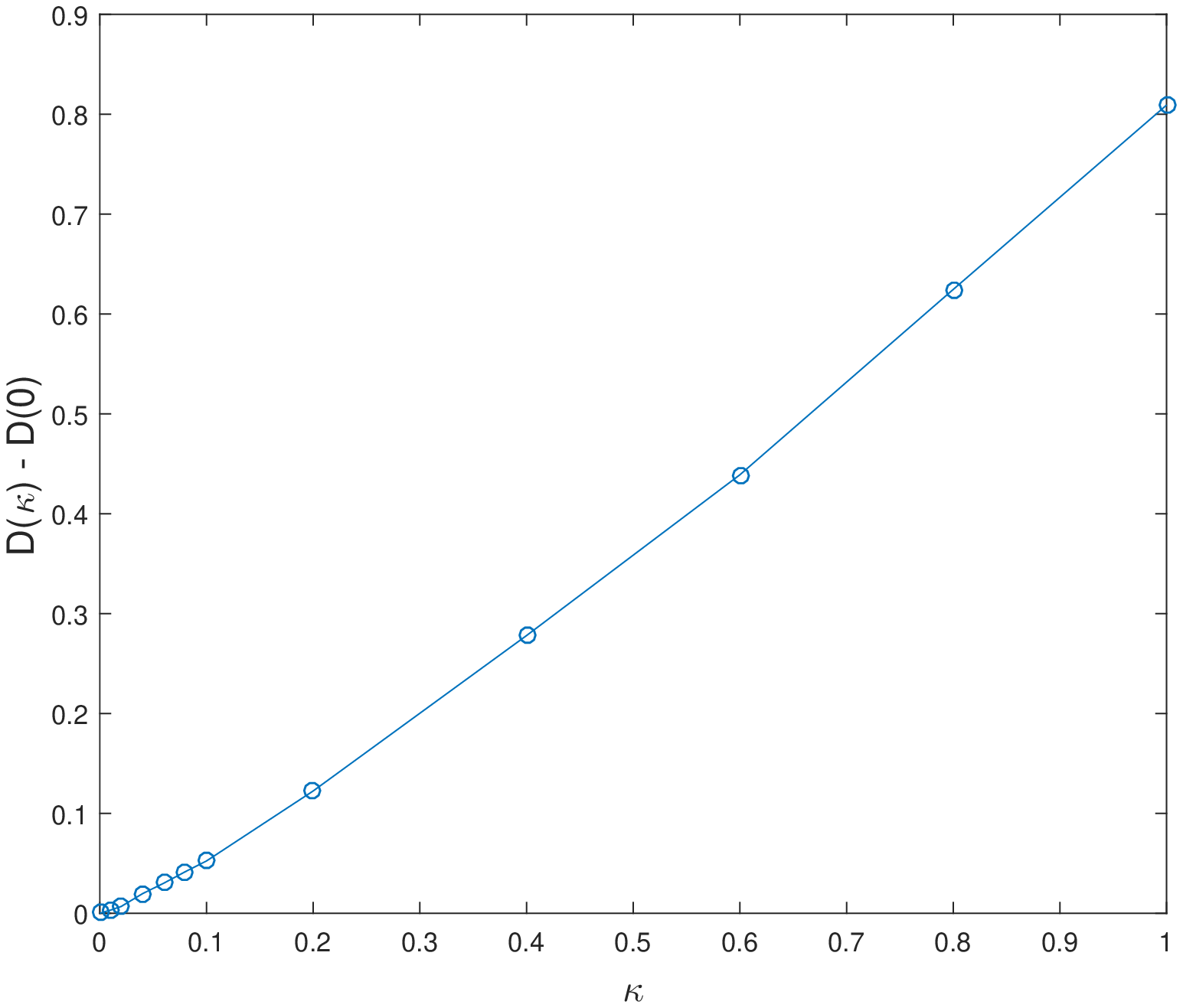}
		\caption{Results for the 3D random flow, $D_0 \approx 0.5685$.}
		\label{DE-to-D0-3D}  
	\end{subfigure}
	\caption{Convergence behaviors of $D^E_{11}(\kappa)$ approaching $D^E_{11}(0)$.}
	\label{DE-to-D0}
\end{figure} 

\section{Conclusion}
\noindent 
In this paper, we studied the numerical homogenization of passive tracer models in random flows. Based on a splitting method, we proposed stochastic structure-preserving schemes to compute the effective diffusivity of the random flows. In addition, we provided rigorous convergence analysis 
for the numerical schemes. Our error analysis is new in the sense that it is based on a probabilistic approach. Specifically, we interpreted the solution process generated by our numerical schemes as a 
Markov process. By using the ergodic theory for the solution process, we proved a sharp error estimate for our numerical schemes in computing the effective diffusivity. Finally, we present numerical results to verify the convergence rate of the proposed method for incompressible random flows both in 2D and 3D spaces. In addition, we observed the exponential decay property and investigated the residual diffusivity phenomenon in the random flows we studied here.

There are two directions we plan to explore in our future work. First, we shall extend the probabilistic approach to provide sharp convergence analysis in computing effective diffusivity for quasi-periodic time-dependent flows. This type of problem is more challenging since the corrector problem does not exist in the $L^2$ space corresponding to the invariant measure. We shall develop other techniques to address this problem.  In addition, we shall investigate the convection-enhanced diffusion phenomenon for general spatial-temporal stochastic flows \cite{Yaulandim:1998,Majda:99} and develop convergence analysis for the corresponding numerical methods.


\section*{Acknowledgement}
\noindent
The research of J. Lyu and Z. Wang are partially supported by the Hong Kong PhD Fellowship Scheme. The research of J. Xin is partially supported by NSF grants DMS-1211179, DMS-1522383, and IIS-1632935. The research of Z. Zhang is supported by Hong Kong RGC grants (Projects 27300616, 17300817, and 17300318), National Natural Science Foundation of China (Project 11601457), Seed Funding Programme for Basic Research (HKU), and Basic Research Programme (JCYJ20180307151603959) of The Science, Technology and Innovation Commission of Shenzhen Municipality. The computations were performed using the HKU ITS research computing facilities that are supported in part by the Hong Kong UGC Special Equipment Grant (SEG HKU09). 

\bibliographystyle{siam}
\bibliography{ZWpaper}
\end{document}